\documentclass[review,hidelinks,onefignum,onetabnum]{siamart220329}
\newtheorem{example}{Example}[section]

\usepackage{lipsum}
\usepackage{amsfonts}
\usepackage{graphicx}
\usepackage{epstopdf}
\usepackage{algorithmic}
\ifpdf
  \DeclareGraphicsExtensions{.eps,.pdf,.png,.jpg}
\else
  \DeclareGraphicsExtensions{.eps}
\fi

\newsiamremark{remark}{Remark}
\newsiamremark{hypothesis}{Hypothesis}
\crefname{hypothesis}{Hypothesis}{Hypotheses}
\newsiamthm{claim}{Claim}

\headers{Two-dimensional greedy randomized Kaczmarz methods for solving large sparse linear systems}{T. Li, M.L. Xiao, X.F. Zhang}

\title{Two-dimensional greedy randomized Kaczmarz methods for solving large-scale linear systems\thanks{Corresponding author: T. Li (tli@hainanu.edu.cn).
\funding{This work was funded by the National Natural Science Foundation of China [grant numbers 12401493 and 12401019], Hainan Provincial Natural Science Foundation of China [grant numbers 122QN214 and 122MS001], and the Academic Programs Project of Hainan University [grant number KYQD(ZR)-21119].}}}

\author{Tao Li\thanks{School of Mathematics and Statistics, Hainan University, Haikou 570228, P. R. China
  (\email{tli@hainanu.edu.cn}).}
  \and Meng-Long Xiao\thanks{School of Mathematics and Statistics, Hainan University, Haikou 570228, P. R. China
  (\email{XML1528155143@163.com}).}
  \and Xin-Fang Zhang \thanks{School of Mathematics and Statistics, Hainan University, Haikou 570228, P. R. China
  (\email{995272@hainanu.edu.cn}).}
}

\usepackage{amsopn}

\ifpdf
\hypersetup{
  pdftitle={Two-dimensional greedy randomized Kaczmarz methods for solving large-scale linear systems},
  pdfauthor={Tao Li, Meng-Long Xiao, Xin-Fang Zhang}
}
\fi

\begin{document}

\maketitle

\begin{abstract}
In this paper, we consider a novel two-dimensional randomized Kaczmarz method and its improved version with simple random sampling, which chooses two active rows with probability proportional to the square of their cross-product-like constant, for solving large-scale linear systems. From the greedy selection strategy with grasping two larger entries of the residual vector at each iteration, we then devise a two-dimensional greedy randomized Kaczmarz method. To improve the above methods further, motivated by the semi-randomized Kaczmarz method and Chebyshev’s law of large numbers, we propose a two-dimensional semi-randomized Kaczmarz method and its modified version with simple random sampling, which is particularly advantageous for big data problems. Theoretically, we prove that the proposed methods converge to the unique least-norm solution of the consistent linear systems. Numerical results on some practical applications illustrate the superiority of the proposed methods compared with some existing ones in terms of computing time.
\end{abstract}


\section{Introduction}\label{section-1}
Consider the iterative solutions of the following large-scale linear systems
\begin{equation}\label{1-1}
A\mathbf{x}=\mathbf{b} ,
\end{equation}
where $A\in \mathbb{C}^{m\times n} $ and $\mathbf{b}\in \mathbb{C}^{m} $ are given, and $\mathbf{x}\in \mathbb{C}^{n} $ is required to be determined. Throughout this paper, we assume that the system is consistent. In this case, we are interested in the least-norm solution
$$
\mathrm{Find}\quad \mathbf{x}_{\star}:={A^\dag }\mathbf{b},\quad s.t.\quad A\mathbf{x}=\mathbf{b}.
$$

Over the past few decades, there has been a significant advancement in theoretical results and iterative methods for Eq.\eqref{1-1}, including Krylov subspace methods \cite{Saad,Saad1,Lanczos,Sonneveld,Vorst}, HSS-type methods \cite{Bai00,Bai01,Bai02} and Kaczmarz-type methods \cite{Kaczmarz,Strohmer,Needell,Elble,Needell2,KDu,Necoara,Ma,Bai,Bai2,D,Miao}. Especially,
the Kaczmarz-type methods have received extensive attention due to their simplicity and efficiency, which require a few rows of the coefficient matrix and no matrix-vector product with less cost. The Kaczmarz method, discovered by Stefan Kaczmarz in 1937 \cite{Kaczmarz}, was later reintroduced and applied in the field of tomography and became known as the Algebraic Reconstruction Technique (ART) \cite{Gordon}. Viewed from the angle of projection,
the Kaczmarz method can be regarded as a one-dimensional oblique projection method.

The original Kaczmarz method, sweeping through the rows of the coefficient matrix cyclically, projects the current point to the hyperplane formed by the active row, which is a row-action method \cite{Censor}. More precisely, if we denote $\mathbf{a}_{i}$ as the $i$-th row of a matrix $A$ and $\mathbf{b}_i$ stands for the $i$-th entry of a vector $\mathbf{b}$, then the iterative scheme of the Kaczmarz method starting with an initial guess $\mathbf{x}_0$ is of the form
\begin{equation}\label{1-2}
\mathbf{x}_{k+1} =\mathbf{x}_{k} + \frac{\mathbf{b}_{i_k}-\mathbf{a}_{i_k}\mathbf{x}_k}{\left | \left | \mathbf{a}_{i_k}  \right |  \right |_2^{2} }\mathbf{a}_{i_k}^{*},
\end{equation}
where ${i_k}= k$ mod $m + 1$ and  the symbol $(\cdot)^*$ denotes the conjugate transpose of the corresponding vector. As seen, the method, updating the current iterate $\mathbf{x}_{k}$ in an inflexible cyclic manner, has a sluggish convergence behavior. Moreover, it is difficult to determine its convergence rate. To address these issues, Strohmer and Vershynin \cite{Strohmer} introduced a randomized Kaczmarz (RK) method by randomly selecting the active or working row of $A$ with probability proportional to its relevance, rather than a given order, which possesses the expected linear convergence rate.

The RK method has gained great attention from many scholars because of its attractive application prospects. To name a few, Needell and Tropp \cite{Needell2} proposed a randomized block Kaczmarz method to improve the computational efficiency of the RK method, for other block Kaczmarz methods, see \cite{KDu,Necoara,Miao}. Then, Ma et al. \cite{Ma} analyzed the convergence properties of the randomized extended Gauss-Seidel and Kaczmarz methods. Hefny et al.  \cite{Hefny}  considered the ridge regression or Tikhonov regularized regression problem, and then presented several variants of the randomized Kaczmarz and randomized Gauss-Siedel (RGS) for searching the iterative solution of Eq.\eqref{1-1}. Loera et al. \cite{Loera} derived a sampling Kaczmarz-Motzkin method with linear feasibility. Moreover, Bai and Wu \cite{Bai} developed the greedy randomized Kaczmarz (GRK) method for obtaining the approximate solutions of Eq.\eqref{1-1}, which converges much faster than the RK solver. After that, they \cite{Bai2}  presented the greedy randomized augmented Kaczmarz method to solve Eq.\eqref{1-1} when it is inconsistent. Besides, Du et al. \cite{Du} developed a Kaczmarz-type inner-iteration preconditioned flexible GMRES method for solving Eq.\eqref{1-1}. Recently, to reduce the computational overhead of the GRK method, Jiang et al. \cite{Jiang} proposed a semi-randomized Kaczmarz method with simple random sampling that improves the efficiency and scalability of the RK method. This method only selects a small fraction of the rows of the matrix $A$ and does not require computing probabilities or constructing an index set for the selected working rows, which is a promising method. Afterward, Li and Wu \cite{LiS} developed a semi‐randomized and augmented Kaczmarz method with simple random sampling for inconsistent linear systems, further expanding the application scope of the Kaczmarz method with simple random sampling.

Observe from the mentioned Kaczmarz-type methods that each iteration of them is a typical one-dimensional oblique projection method. Specifically, these methods, at $k$-th iteration, extract an approximate solution $\mathbf{x}_{k+1}$ from a one-dimensional affine subspace $\mathbf{x}_{k}+\mathcal K_{0}$ by imposing the Petrov-Galerkin condition
\begin{equation}\label{1-3}
{\mathbf{b}}-{A}\mathbf{x}_{k+1} \perp \mathcal L_{0},
\end{equation}
where $\mathcal K_{0}= \mathrm{span}\left \{\mathbf{a} _{i_k}^* \right \}$, and $\mathcal L_{0}=\mathrm{span}\left \{ \mathbf{u} _{i_k}  \right \}$  is another one-dimensional subspace with $\mathbf{u} _{i_k} $ being the $i_k$-th standard basis vector. Generally, the solution $\mathbf{x}_{k+1}$, generated by a two-dimensional or even higher dimensional subspace from the active rows of $A$, may approach the exact solution faster than the former. To the best of our knowledge, there is comparatively little literature on Kaczmarz-type methods that employ the multiple active rows of $A$ simultaneously at each iteration for solving Eq.\eqref{1-1}. Only Needell and Ward \cite{Needell3} proposed a two-subspace randomized Kaczmarz method, which is a block RK method with two rows being chosen from each block submatrix, and Wu \cite{WWT} developed a generalized two-subspace Kaczmarz (GTRK) method. Nevertheless, it is remarkable that one becomes impractical when the dimension of $\mathcal K_{0}$ is large since the storage and computational operations required by $\mathbf{x}_{k+1}$ are costly. To speed up the convergence of the traditional Kaczmarz method with a reasonable cost, we first propose a novel two-dimensional randomized Kaczmarz (TRK) method and its improved version with simple random sampling for solving Eq.\eqref{1-1}. The main idea of the proposed methods is to seek the approximate solution $\mathbf{x}_{k+1}$ from a two-dimensional affine subspace $\mathbf{x}_{k}+\mathcal K_{1}$ by imposing the Petrov condition
\begin{equation}\label{1-4}
\mathbf{b}-{A}\mathbf{x}_{k+1}  \perp \mathcal L_{1},
\end{equation}
where $\mathcal K_{1}= \mathrm{span}\left \{\mathbf{a} _{i_k}^*,\mathbf{a} _{j_k}^*\right \}$, and $\mathcal L_{1}=\mathrm{span}\left \{ \mathbf{u} _{i_k},\mathbf{u} _{j_k}  \right \}$. From the projection viewpoint, the next iterate $\mathbf{x}_{k+1}$ is obtained by projecting orthogonally the current iterate $\mathbf{x}_{k}$ onto the intersection of two unparallel hyperplanes, defined as $\mathbf{a}_{i_k}{\mathbf{x}}=\mathbf{b}_{i_k}$ and $\mathbf{a}_{j_k}{\mathbf{x}}=\mathbf{b}_{j_k}$, that is, a subspace of dimension $n-2$. By taking advantage of the two working rows of $A$, the TRK method converges much faster than the RK solver, particularly when the solution 
yields the multiple equations of Eq.\eqref{1-1} simultaneously.

As stated earlier, each iteration of the TRK method is surely a two-dimensional oblique projection method. In this paper, we first devise the two-dimensional Kaczmarz method, based on the Petrov–Galerkin condition and given order, for solving Eq.\eqref{1-1}. Moreover, it was proved in \cite{Strohmer}  that the randomized Kaczmarz method, sweeping through the rows in random order, vastly outperforms the Kaczmarz method with a given order.
Hence, we then propose the two-dimensional randomized Kaczmarz (TRK) method, which chooses two active rows of $A$ with probability proportional to the square of their cross-product-like constant.
Note that, for the TRK method, we have to scan all the rows of the data matrix in advance, and compute probabilities of all two rows, which is uncompetitive for big data problems. To overcome these issues, motivated by Chebyshev's (weak) law of large numbers and the semi-randomized Kaczmarz method \cite{Jiang}, we develop a practical two-dimensional randomized Kaczmarz method with simple random sampling (TRKS), which converges much faster than that of TRK method in terms of the computing time and iterations. 

Selecting effective working rows from the coefficient matrix plays a key role in fast solving Eq.\eqref{1-1}. Similar to RK, the weakness of the TRK solver is as follows: if the matrix $A$ is actively scaled with a diagonal matrix that normalizes the Euclidean norm of all its rows to the same constant, the selection criterion becomes equivalent to uniform sampling. We propose a different, yet more effective probability criterion based on a greedy selection strategy that focuses on the larger two entries of the residual vector. Based on this criterion, we then construct a two-dimensional greedy randomized Kaczmarz (TGRK) method for solving Eq.\eqref{1-1}. We theoretically prove the convergence of this method and provide an upper bound of its convergence rate. Furthermore, this method requires computing the residuals to determine the probabilities, which requires linear combinations of all columns of the data matrix. To deal with this drawback, inspired by \cite{Jiang}, we present a two-dimensional semi-randomized Kaczmarz method, which only needs to seek the two rows corresponding to the largest and second-largest elements of the residual vectors at each iteration. From Chebyshev's (weak) law of large numbers, we further develop a two-dimensional semi-randomized Kaczmarz method with simple random sampling.

The organization of this paper is as follows. In Section \ref{section-2}, we briefly review some basic results and some Kaczmarz-type methods that are helpful in the later sections. In Section \ref{section-3}, we propose several two-dimensional Kaczmarz-type methods, for solving Eq.\eqref{1-1}. The convergence of which is also established. In Section \ref{section-4}, we present some numerical examples to demonstrate the feasibility and validity of the proposed methods compared with some existing methods. Finally, in Section \ref{section-5}, we conclude this paper by giving some remarks.

\section{Preliminaries}\label{section-2}
In this section, we briefly review some basic knowledge and popular Kcazmarz-type methods, including the GRK method, the GTRK method, the semi-randomized Kaczmarz method, and its variants with simple random sampling. As was shown in \cite{Bai},
the least-squares solution of Eq.\eqref{1-1} can be expressed as 
\begin{equation}\label{2-1}
\begin{aligned}
\mathbf{x}_{LS}= \arg \mathop {\min }\limits_{\mathbf{x}\in\mathbb{C}^n} \parallel \mathbf{b}-A\mathbf{x}\parallel _2 \equiv {A^\dag }\mathbf{b} +(I - P)\mathbf{z},\quad \forall \mathbf{z} \in {\mathbb{C}^n} ,
\end{aligned}
\end{equation}
where $A^\dag $ denotes the Moore-Penrose pseudoinverse of $A$, and $P=A^\dag A$ is the orthogonal projection onto the range space of $A^*$. Observe that $(I - P)$ is the projection matrix from $\mathbb{C}^n$ to the zero space of $A$. Indeed, $\mathbf{x}_{\star } =A^{\dag}\mathbf{b}$ is a solution of Eq.\eqref{1-1}, i.e., the least Euclidean-norm solution, when it is consistent. The reason is that the solution $\mathbf{x} $ of Eq.\eqref{1-1} can be expressed as $\mathbf{x}=\mathbf{x}_{\star }+\mathbf{y}$ $(\mathbf{x}_{\star }\ne \mathbf{y})$, where 
$\mathbf{y}$ is a vector belonging to the zero space of matrix $A$, i.e., $\mathbf{y}\in \mathrm{null} (A)$. 
Then it is ready to see that $\left \langle \mathbf{x}_{\star }, \mathbf{y} \right \rangle =0$. Hence, it follows that $\left \| \mathbf{x} \right \|^{2}_{2} = \left \|   \mathbf{x}_ \star \right \|^{2}_{2}  + \left \| \mathbf{y} \right \|  ^{2}_{2}$, 
which means that $\mathbf{x}_\star$ is the least-squares solution as claimed.

Given a matrix $A:=[a_{ij}]\in \mathbb{C}^{m\times n}$, its Frobenius norm, $2$-norm and $l_{2,1}$-norm \cite{Saad1} are respectively defined as 
\begin{equation}\label{2-2}
\| A\|_F:= {(\sum\limits_{j = 1}^n {\sum\limits_{i = 1}^m {{{\left| {{a_{ij}}} \right|}^2}} } )^{\frac{1}{2}}}=\sqrt{\sum_{i=1}^{t}\sigma_i^2},\quad 
\|A \|_2:= \sigma_{s},\quad \| A\|_{2,1}:= \sum_{i=1}^{m}\|\mathbf{a}_i\|_2, 
\end{equation} 
where $\sigma_{i}$ denotes the nonzero singular values of $A$, and $\sigma_{s}$ is the largest nonzero singular value. For a vector $\mathbf{x}:=[x_i]\in \mathbb{C}^n$, its $1$-norm is of the form $\|\mathbf{x}\|_1=\sum_{i=1}^{n}|x_i|$.  Given two vectors $\mathbf{x}:=[x_i]$ and $\mathbf{y}:=[y_i]\in \mathbb{C}^n$, their cross-product-like constant is defined as
\begin{equation}\label{2-3}|\mathbf{x}\times \mathbf{y}|=\|\mathbf{x}\|_2\|\mathbf{y}\|_2|\sin\theta|=\sqrt{\|\mathbf{x}\|_2^2\|\mathbf{y}\|_2^2-|\mathbf{x}\mathbf{y}^*|^2},\end{equation}
where $\theta$ is the angle between $\mathbf{x} $ and $\mathbf{y}$.
Denote by $\mathbb{E} _{k}$, the expected value of the first $k$ iterations \cite{Bai}, i.e.,
$\mathbb{E} _{k} [\cdot ]=\mathbb{E}[\cdot \mid i_{0} ,i_{1},\dots ,i_{k-1}  ],$ where $i_{j} (j=0,1,\dots,k-1 )$ is the $j$-th row chosen at the $j$-th iterate, and from the law of iterated expectations, we obtain $\mathbb{E}[\mathbb{E} _{k} [\cdot ]]=\mathbb{E}[\cdot]$. 

In what follows, we first recall the GRK method \cite{Bai} with a new probability criterion, which selects the working rows by grasping larger entries of the residual vector.
\begin{algorithm}
	\caption{GRK \cite{Bai}}\label{algo-2-1}
	\begin{algorithmic}[1] 
        \REQUIRE $A,  \mathbf{b},  \mathbf{x}_{0}$\\
        \ENSURE $\mathbf{x}_{k+1}$\\
        \STATE
       For $k=0, 1, 2,\cdots $, until convergence, do:\\
        \STATE
        Compute\\
        $$
         \epsilon  _{k} = \frac{1}{2}(\frac{1}{{\parallel \mathbf{b} - A\mathbf{x}_{k}{\parallel _{2}^2}}}\mathop {\max }\limits_{i_k \in [m]} \left\{ \frac{|\mathbf{b}_{i_k}-\mathbf{a}_{i_k}\mathbf{x}_k|^2}{\|\mathbf{a}_{i_k}\|_2^2}\right\} + \frac{{1}}{{\parallel A\parallel _F^2}}).
         $$\\
 \STATE
Determine the index set of positive integers\\
$$\mathcal{U}_k = \left\{ {{i_k}\left| {|\mathbf{b}_{i_k}-\mathbf{a}_{i_k}\mathbf{x}_k|^2 \ge {\epsilon_k}} \right.\parallel \mathbf{b} - 
A{\mathbf{x}_{k}}{\parallel _{2}^2}}\|\mathbf{a}_{i_k}\|_2^2 \right\}.$$\\
 \STATE
   Compute the $i$-th element of the vector $\tilde{\mathbf{r}} _{k}$ by  \\
$$\tilde{\mathbf{r}} _{k} ^{(i)}=\left\{ \begin{array}{l}
{\mathbf{b}_i} - {\mathbf{a}_i}{\mathbf{x}_{k}},\qquad  \mathrm{if} \quad i \in \mathcal{U}_{k}
\\
0 ,\qquad \qquad \quad \enspace \mathrm{otherwise}
\end{array}
\right.$$
    \STATE   
Select ${i_k} \in {\mathcal{U}_k}$ with probability $\mathbb{P}$(row =${i_k}$)=$ \mid \tilde{\mathbf{r}} _{k} ^{(i_k)}\mid ^{2} /\left \| \tilde{\mathbf{r}}_k \right \|_{2} ^{2} $.\\
    \STATE
 Compute $\mathbf{x}_{k+1} =\mathbf{x}_{k} +\frac{\mathbf{b}_{i_{k}} -\mathbf{a}_{i_{k}}\mathbf{x}_{k}}{\left | \left | \mathbf{a}_{i_{k}}  \right |  \right |_{2} ^{2} }\mathbf{a}_{i_{k}} ^{*}$.
	\end{algorithmic}
\end{algorithm}

By choosing two rows of $A$ with probability proportional to the relevance at each iteration, Wu \cite{WWT} then presented the GTRK method as follows. 

\begin{algorithm}[H]
	\caption{GTRK \cite{WWT} }
	\label{algo-2-2}
	\begin{algorithmic}[1]
        \REQUIRE $A,  \mathbf{b},  \mathbf{x}_{0}$\\
        \ENSURE $\mathbf{x}_{k+1}$\\
        \STATE
       For $k=0, 1, 2,\cdots $, until convergence, do:\\
            \STATE Select $i_k\in\{1,2,\dots,m\} $ with probability Pr(row=$i_k$)=$\frac{\left\|\mathbf{a}_{i_k}\right\|_2^2}{\|A\|_{F}^2}$.
            \STATE Select $j_k\in\{1,2,\dots,m\} \setminus  \left \{ i_k \right \}  $ with probability Pr(row=$j_k$)=$\frac{\left\|\mathbf{a}_{j_k}\right\|_2^2}{\|A\|_{F}^2-\|\mathbf{a}_{i_k}\|_2^2}$.
            \STATE Set $\mu_k=\frac{\mathbf{a}_{j_k}\mathbf{a}_{i_k}^{*}}{\|\mathbf{a}_{j_k}\|_2\|\mathbf{a}_{i_k}\|_2}$,  $\tilde{r} _{k_1}=\frac{\mathbf{b}_{i_k}-\mathbf{a}_{i_k}\mathbf{x}_k}{\|\mathbf{a}_{i_k}\|_2}$ and $ \tilde{r} _{k_2}=\frac{\mathbf{b}_{j_k}-\mathbf{a}_{j_k}\mathbf{x}_k}{\|\mathbf{a}_{j_k}\|_2}$.
           \STATE Set $\mathbf{x}_{k+1}=\mathbf{x}_{k}+\frac{\tilde{r} _{k_1}-\bar{\mu} _k\tilde{r} _{k_2}}{(1-|\mu _k|^2)\|\mathbf{a}_{i_k}\|_2}\mathbf{a}_{i_k}^{*}+\frac{\tilde{r} _{k_2}-\mu _k\tilde{r} _{k_1}}{(1-|\mu _k|^2)\|\mathbf{a}_{j_k}\|_2}\mathbf{a}_{j_k}^{*}$.
	\end{algorithmic}  
\end{algorithm}

It was proved from \cite{Bai} that the GRK method converges much faster than the RK method. However, the GRK has to compute the residual and implement a large amount of flops to construct an index set at each iteration. This requires the whole information of the coefficient matrix, which is unavailable when the size of the matrix is huge, especially for big data problems. Additionally, the probability computation of the working row in GRK may suffer from inaccurate probability operations due to the rounding-off errors.
To address these issues, Jiang and Wu \cite{Jiang} introduced a semi-randomized Kaczmarz (SRK) method, choosing the largest element in the magnitude of the (relative) residual to determine the working row, which does not require explicit probability operations. To further improve its performance, inspired by Chebyshev's (weak) law of large numbers, they derived a semi-randomized Kaczmarz method with a simple random sampling strategy (SRKS) for solving Eq.\eqref{1-1} as follows.

\begin{algorithm}
	\caption{SRK \cite{Jiang}}
	\label{algo-2-3}
	\begin{algorithmic}[1]
        \REQUIRE $A,  \mathbf{b},  \mathbf{x}_{0}$\\
        \ENSURE $\mathbf{x}_{k+1}$\\
        \STATE
       For $k=0, 1, 2,\cdots $, until convergence, do:\\
            \STATE Select $i_k\in\{1,2,\dots,m\} $ satisfying $$\frac{|\mathbf{b}_{i_k}-\mathbf{a}_{i_k}\mathbf{x}_k|}{\|\mathbf{a}_{i_k}\|_2}=\max_{1\leq i \leq m}\frac{|\mathbf{b}_i-\mathbf{a}_{i}\mathbf{x}_k|}{\|\mathbf{a}_i\|_2}.$$

    \STATE
 Compute $\mathbf{x}_{k+1} =\mathbf{x}_{k} +\frac{\mathbf{b}_{i_{k}} -\mathbf{a}_{i_{k}}\mathbf{x}_{k}}{\left | \left | \mathbf{a}_{i_{k}}  \right |  \right |_{2} ^{2} }\mathbf{a}_{i_{k}} ^{*}.$

	\end{algorithmic}  
\end{algorithm}

\begin{algorithm}
	\caption{SRKS \cite{Jiang}}
	\label{algo-2-4}
	\begin{algorithmic}[1]
        \REQUIRE $A,  \mathbf{b},  \mathbf{x}_{0}$\\
        \ENSURE $\mathbf{x}_{k+1}$\\
        \STATE
       For $k=0, 1, 2,\cdots $, until convergence, do:\\
               \STATE
        Generate a set of indicators $\Phi_k$, i.e., choosing $\eta m$ rows of $A$ by the simple random sampling, where $0<\eta<1$.
            \STATE Select $i_k\in\Phi_k $ satisfying $$\frac{|\mathbf{b}_{i_k}-\mathbf{a}_{i_k}\mathbf{x}_k|}{\|\mathbf{a}_{i_k}\|_2}=\max_{i \in \Phi_k}\frac{|\mathbf{b}_i-\mathbf{a}_{i}\mathbf{x}_k|}{\|\mathbf{a}_i\|_2}.$$

    \STATE
 Compute $\mathbf{x}_{k+1} =\mathbf{x}_{k} +\frac{\mathbf{b}_{i_{k}} -\mathbf{a}_{i_{k}}\mathbf{x}_{k}}{\left | \left | \mathbf{a}_{i_{k}}  \right |  \right |_{2} ^{2} }\mathbf{a}_{i_{k}} ^{*}$.
	\end{algorithmic}  
\end{algorithm}

\section{Two-dimensional randomized Kaczmarz-type methods}\label{section-3}

In this section, we will devise the two-dimensional randomized Kaczmarz method and its variants for solving the linear systems \eqref{1-1}. First, a two-dimensional randomized Kaczmarz (TRK) method, built upon the Petrov-Galerkin condition and the cross-product-like constant of the active rows, is proposed. From the simple random sampling strategy, we present an improved TRK by only utilizing a small portion of rows of the coefficient matrix, which performs much faster than TRK. Secondly, a two-dimensional greedy randomized Kaczmarz method,
based on the greedy selection strategy with grasping two larger entries of the residual vector at each iteration, is developed for solving Eq.\eqref{1-1}. Thirdly, a two-dimensional semi-randomized Kaczmarz (TSRK) method, by selecting the rows corresponding to the current largest and the second largest homogeneous residuals, is proposed.
The method does not need to construct index sets
with working rows, nor compute probabilities. To further save the arithmetic operations of TSRK, we give a modified TSRK method with simple random sampling for solving Eq.\eqref{1-1}, which only computes some elements of the residual vectors corresponding to the simple sampling set.



\subsection{Two-dimensional randomized Kaczmarz methods}\label{subsection-3-1}

As stated earlier, each iterate of the Kaczmarz is a typical one-dimensional projection method. Note that an approximate solution given by a suitable two-dimensional subspace may converge to $\mathbf{x}_{\star }$ of \eqref{1-1} faster than one generated by the subspace $\mathcal K_{0}$. Motivated by this idea and Eq.\eqref{1-3}, we will present a two-dimensional Kaczmarz method for solving Eq.\eqref{1-1}.  Let $\mathcal{K}_1= \mathrm{span}\left \{\mathbf{a} _{i}^{*}, \mathbf{a} _{j}^{*} \right \}$, and $\mathcal{L}_1=\mathrm{span}\left \{\mathbf{u} _{i}, \mathbf{u}  _{j}  \right \}$ be a search space and a constraint space, respectively, where $\mathbf{u} _{i}, \mathbf{u} _{j} $ are the $i$-th and $j$-th standard basis vectors.  We find that an approximate solution $\mathbf{x}_{k+1}$ yielding the following condition:
\begin{equation}\label{3-1}
\mbox{Find}\ \mathbf{x}_{k+1}\in  \mathbf{x}_{k}+\mathcal K_{1},\ \mbox{such that}\ \mathbf{b}-{A}\mathbf{x}_{k+1} \perp \mathcal L_{1}. 
\end{equation}
This shows that $\mathbf{x}_{k+1}$ can be expressed as
\begin{equation}\label{3-2}
\mathbf{x}_{k+1} =\mathbf{x}_{k} + \gamma_k\mathbf{a} _{i}^{*} + \lambda_k\mathbf{a} _{j} ^{*},
\end{equation}
where $\gamma_k$ and $\lambda_k$ are scalars. 
Here, if $\mathbf{a}_{i}$ is parallel to $\mathbf{a}_{j}$, then Eq.\eqref{3-2} reduces to
\begin{equation}\label{3-26}
\mathbf{x}_{k+1} =\mathbf{x}_{k} +\frac{\mathbf{b}_{i} -\mathbf{a}_{i}\mathbf{x}_{k}}{\left | \left | \mathbf{a}_{i}  \right |  \right |_{2} ^{2} }\mathbf{a}_{i} ^{*},
\end{equation}
otherwise, it follows from Eq.\eqref{3-1} that
$$\left\{ \begin{array}{l}
 \langle 
\mathbf{b}-A\mathbf{x}_{k+1} ,\mathbf{u} _{i}  \rangle =0,\\
 \langle 
\mathbf{b}-A\mathbf{x}_{k+1} ,\mathbf{u} _{j}   \rangle =0,
 \end{array}
\right.$$
that is,
$$
\left\{ \begin{array}{l}
\mathbf{b} _i-\mathbf{a}_{i}\mathbf{x}_k-\mathbf{a}_i\mathbf{a}_i^*\gamma_k-\mathbf{a} _i\mathbf{a} _j^*\lambda_k=0,\\
\mathbf{b}_j-\mathbf{a}_{j}\mathbf{x}_k-\mathbf{a}_j\mathbf{a}_i^*\gamma_k-\mathbf{a}_j\mathbf{a}_j^*\lambda_k=0.
\end{array}
\right.$$
Therefore, the parameters $\gamma_k$ and $\lambda_k$ 
are given by
$$
\begin{aligned}
&\gamma_k=\frac{\|\mathbf{a}_j\|_2^2(\mathbf{b}_i-\mathbf{a}_{i}\mathbf{x}_k)-\mathbf{a}_i\mathbf{a}_j^*(\mathbf{b}_j-\mathbf{a}_{j}\mathbf{x}_k)}{\|\mathbf{a}_i\|_2^2\|\mathbf{a}_j\|_2^2-|\mathbf{a}_{i}\mathbf{a}_{j}^*|^2},\\
&\lambda_k=\frac{\|\mathbf{a}_i\|_2^2(\mathbf{b}_j-\mathbf{a}_{j}\mathbf{x}_k)-\mathbf{a}_j\mathbf{a}_i^*(\mathbf{b}_i-\mathbf{a}_{i}\mathbf{x}_k)}{\|\mathbf{a}_i\|_2^2\|\mathbf{a}_j\|_2^2-|\mathbf{a}_{i}\mathbf{a}_{j}^*|^2}.
\end{aligned}
$$

From the above results, a two-dimensional Kaczmarz method for Eq.\eqref{1-1}, sweeping through the rows of $A$ in a cyclic manner, is as follows.
\begin{algorithm}
	\caption{Two-dimensional Kaczmarz method}\label{algo-3-1}
	\begin{algorithmic}[1] 
        \REQUIRE $A,  \mathbf{b},  \mathbf{x}_{0}$\\
        \ENSURE $\mathbf{x}_{k+1}$\\
        \STATE
        For $k=0, 1, 2,\cdots $, until convergence, do:\\
        \STATE
        Select $i_k =2( k$ mod $m) + 1 $, $j_k=i+1$ $(i_k,j_k\leq m)$
        \\
 \STATE  If $\mathbf{a}_{i_k}$ is parallel to $\mathbf{a}_{j_k}$, then we compute
$$
\mathbf{x}_{k+1} =\mathbf{x}_{k} +\frac{\mathbf{b}_{i_{k}} -\mathbf{a}_{i_{k}}\mathbf{x}_{k}}{\left | \left | \mathbf{a}_{i_{k}}  \right |  \right |_{2} ^{2} }\mathbf{a}_{i_{k}} ^{*},
$$
otherwise, compute
$$\begin{aligned}
\gamma_k&=\frac{\|\mathbf{a}_{j_k}\|_2^2(\mathbf{b}_{i_k}-\mathbf{a}_{i_k}\mathbf{x}_k)-\mathbf{a}_{i_k}\mathbf{a}_{j_k}^*(\mathbf{b}_{j_k}-\mathbf{a}_{j_k}\mathbf{x}_k)}{\|\mathbf{a}_{i_k}\|_2^2\|\mathbf{a}_{j_k}\|_2^2-|\mathbf{a}_{i_k}\mathbf{a}_{j_k}^*|^2},\\
\lambda_k&=\frac{\|\mathbf{a}_{i_k}\|_2^2(\mathbf{b}_{j_k}-\mathbf{a}_{j_k}\mathbf{x}_k)-\mathbf{a}_{j_k}\mathbf{a}_{i_k}^*(\mathbf{b}_{i_k}-\mathbf{a}_{i_k}\mathbf{x}_k)}{\|\mathbf{a}_{i_k}\|_2^2\|\mathbf{a}_{j_k}\|_2^2-|\mathbf{a}_{i_k}\mathbf{a}_{j_k}^*|^2}.
\end{aligned}$$
 Update $\mathbf{x}_{k+1} =\mathbf{x}_{k} + \gamma_k\mathbf{a}_{i_k}^{*} + \lambda_k\mathbf{a}_{j_k} ^{*}$.
	\end{algorithmic}
\end{algorithm}

From a purely algebraic point of view, the two-dimensional Kaczmarz iteration is simple and concise. Actually, each iteration in \eqref{3-2} is formed by orthogonally projecting the current point $\mathbf{x}_{k}$ onto two unparallel hyperplanes in a cyclic order. 

\begin{rem}\label{Rem-1}
\rm Note that the iterative scheme of Algorithm \ref{algo-3-1} appears in a recent paper \cite{WWT}, in which a generalized two-subspace randomized Kaczmarz (GTRK) method was discussed. As shown, from an algebraic point of view, we only utilize the Petrov–Galerkin conditions \cite{Saad}
to propose Algorithm \ref{algo-3-1} with a new cyclic order, which is completely different from GTRK and performs more straightforwardly.
\end{rem}

In what follows, we present the following theorem which guarantees the convergence of Algorithm \ref{algo-3-1}.

\begin{theorem}\label{The-1}
{\rm Let $\mathbf{x}_{\star}=A^{\dagger}\mathbf{b}$ be the solution of Eq.\eqref{1-1}. Then the iterative sequence $\{\mathbf{x}_{k}\}$ generated by Algorithm \ref{algo-3-1} converges to $\mathbf{x}_{\star }$ for any initial vector $\mathbf{x}_0$.
}
\end{theorem}
\begin{proof} 
Define
$$\begin{aligned}\Theta_{k} &=\mathbf{a}_i\mathbf{a}_j^* \mathbf{a}_i^* \mathbf{a}_j-\|\mathbf{a}_j\|_2^2\mathbf{a}_i^* \mathbf{a}_i+\mathbf{a}_j\mathbf{a}_i^* \mathbf{a}_j^* \mathbf{a}_i-\|\mathbf{a}_i\|_2^2
\mathbf{a}_j^* \mathbf{a}_j,\\
\varphi_k&=\|\mathbf{a}_j\|_2^2|\mathbf{a}_{i}\overline{\mathbf{x}}_k|^2+\|\mathbf{a}_i\|_2^2|\mathbf{a}_{j}\overline{\mathbf{x}}_k|^2-(\mathbf{a}_i\overline{\mathbf{x}}_k)^* \mathbf{a}_i\mathbf{a}_j^* \mathbf{a}_j\overline{\mathbf{x}}_k-(\mathbf{a}_j\overline{\mathbf{x}}_k)^* \mathbf{a}_j\mathbf{a}_{i}^* \mathbf{a}_i\overline{\mathbf{x}}_k,\end{aligned}$$
in which $\overline{\mathbf{x}}_k=\mathbf{x}_k-\mathbf{x}_{\star}$.
From Algorithm \ref{algo-3-1}, we have
\begin{equation}\label{3-3}
\begin{aligned}
\overline{\mathbf{x}}_{k+1}&=\overline{\mathbf{x}}_k+\frac{\Theta_k\overline{\mathbf{x}}_k}{\|\mathbf{a}_i\|_2^2\|\mathbf{a}_j\|_2^2-|\mathbf{a}_i\mathbf{a}_j^*|^2}\\
&=\overline{\mathbf{x}}_k-P\overline{\mathbf{x}}_k\\
&=(I-P)\overline{\mathbf{x}}_k,   
\end{aligned}
\end{equation}
where $P=-\frac{\Theta_k}{\|\mathbf{a}_i\|_2^2\|\mathbf{a}_j\|_2^2-|\mathbf{a}_i\mathbf{a}_j^*|^2}$. It is easy to verify that $P^2=P$ and $P^*=P$. This fact shows that $P$ is a projection matrix, and if $P$ is a projector, then so is $(I-P)$. From Eq.\eqref{3-3}, it follows that
\begin{equation}\label{3-4}
\begin{aligned}
\|\overline{\mathbf{x}}_{k+1}\|_2^2
&=\|(I-P)\overline{\mathbf{x}}_{k}\|_2^2\\
&=\langle(I-P)\overline{\mathbf{x}}_{k},(I-P)\overline{\mathbf{x}}_{k}\rangle\\
&=\overline{\mathbf{x}}_k^*(I-P)\overline{\mathbf{x}}_k\\
&=\overline{\mathbf{x}}_k^*\left(\overline{\mathbf{x}}_k-P\overline{\mathbf{x}}_k\right)\\
&=\|\overline{\mathbf{x}}_k\|_2^2-\langle P\overline{\mathbf{x}}_k,P\overline{\mathbf{x}}_k\rangle\\
&=\|\overline{\mathbf{x}}_k\|_2^2-\|
P\overline{\mathbf{x}}_k\|_2^2\\
&=\|\overline{\mathbf{x}}_k\|_2^2-\|\mathbf{x}_{k+1}-\mathbf{x}_{k}\|_2^2.
\end{aligned}  
\end{equation}
Then for the non-parallel case, we have
$$
\begin{aligned}
\|\overline{\mathbf{x}}_{k+1}\|_2^2&=\|\overline{\mathbf{x}}_k\|_2^2-\|\mathbf{x}_{k+1}-\mathbf{x}_{k}\|_2^2\\
&=\|\overline{\mathbf{x}}_k\|_2^2-\frac{\varphi_k}{\|\mathbf{a}_i\|_2^2\|\mathbf{a}_j\|_2^2-|\mathbf{a}_i\mathbf{a}_j^*|^2}\\
&\leq \|\overline{\mathbf{x}}_k\|_2^2-\frac{2\|\mathbf{a}_i\|_2\|\mathbf{a}_j\|_2|\mathbf{a}_{i}\overline{\mathbf{x}}_k||\mathbf{a}_{j}\overline{\mathbf{x}}_k|-2|\mathbf{a}_{i}\overline{\mathbf{x}}_k||\mathbf{a}_{j}\overline{\mathbf{x}}_k||\mathbf{a}_i\mathbf{a}_j^*|}{\|\mathbf{a}_i\|_2^2\|\mathbf{a}_j\|_2^2-|\mathbf{a}_i\mathbf{a}_j^*|^2}\\
&=\|\overline{\mathbf{x}}_k\|_2^2-\frac{2|\mathbf{a}_{i}\overline{\mathbf{x}}_k||\mathbf{a}_{j}\overline{\mathbf{x}}_k|(\|\mathbf{a}_i\|_2\|\mathbf{a}_j\|_2-|\mathbf{a}_i\mathbf{a}_j^*|)}{\|\mathbf{a}_i\|_2^2\|\mathbf{a}_j\|_2^2-|\mathbf{a}_i\mathbf{a}_j^*|^2}\\
&\leq\|\overline{\mathbf{x}}_k\|_2^2-\frac{|\mathbf{a}_{i}\overline{\mathbf{x}}_k||\mathbf{a}_{j}\overline{\mathbf{x}}_k|}{\|\mathbf{a}_i\|_2\|\mathbf{a}_j\|_2},
\end{aligned}
$$
otherwise, for the parallel case, we have
$$
\begin{aligned}
\|\overline{\mathbf{x}}_{k+1}\|_2^2&=\|\overline{\mathbf{x}}_k\|_2^2-\|\mathbf{x}_{k+1}-\mathbf{x}_{k}\|_2^2\\
&=\|\overline{\mathbf{x}}_k\|_2^2-\frac{|\mathbf{a}_{i}\overline{\mathbf{x}}_k|^2}{\|\mathbf{a}_i\|_2^2}.
\end{aligned}
$$
Consequently, for any $i,j \in[m]$, it follows that
\begin{equation}\label{3-27}
\begin{aligned}
\|\overline{\mathbf{x}}_{k+1}\|_2^2&=\|\overline{\mathbf{x}}_k\|_2^2-\|\mathbf{x}_{k+1}-\mathbf{x}_{k}\|_2^2\\
&\leq\|\overline{\mathbf{x}}_k\|_2^2-\frac{|\mathbf{a}_{i}\overline{\mathbf{x}}_k||\mathbf{a}_{j}\overline{\mathbf{x}}_k|}{\|\mathbf{a}_i\|_2\|\mathbf{a}_j\|_2}.
\end{aligned}
\end{equation}
As seen, $\frac{|\mathbf{a}_{i}\overline{\mathbf{x}}_k||\mathbf{a}_{j}\overline{\mathbf{x}}_k|}{\|\mathbf{a}_i\|_2\|\mathbf{a}_j\|_2}\ge0$ holds for any $k$,
and $\mathbf{x}_{k}$ is the exact solution $\mathbf{x}_{\star}$ once it equals to $0$. Consequently, the desired result is that the error norm obtained by $\overline{\mathbf{x}}_k$ decreases until $0$. 
\end{proof}
\begin{rem}
\rm The above theorem illustrates that, all cases, including the non-parallel and parallel cases, yield Eq.\eqref{3-27}. Therefore, in the following we only consider the non-parallel case. Moreover, we can readily to reselect two non-parallel working rows once the selected two rows are parallel.
\end{rem}

As was shown in \cite{Strohmer}, the RK method,  selecting each row of $A$ in a random order, rather than a given order, vastly outperforms the original Kaczmarz method. Therefore, we propose a two-dimensional randomized Kaczmarz (TRK) method, which chooses two working rows of $A$ with probability proportional to the square of their cross-product-like constant. In exact arithmetic, the TRK method is as follows.
\begin{algorithm}
	\caption{Two-dimensional randomized Kaczmarz method}\label{algo-3-2}
	\begin{algorithmic}[1] 
        \REQUIRE $A,  \mathbf{b},  \mathbf{x}_{0}$\\
        \ENSURE $\mathbf{x}_{k+1}$\\
        \STATE
        For $k=0, 1, 2,\cdots $, until convergence, do:\\
        \STATE
        Select ${i_k,j_k} \in \left \{ 1,2,\dots ,m \right \} $ with probability $\mathbb{P}$(rows =$i_k,j_k$)=$ \frac{|\mathbf{a}_{i_k}\times \mathbf{a}_{j_k}|^2}{\left \| A \right \|_{F } ^{4}-\|AA^*\|_F^2  } $
        \\
 \STATE
Compute
$$\begin{aligned}
\gamma_k&=\frac{\|\mathbf{a}_{j_k}\|_2^2(\mathbf{b}_{i_k}-\mathbf{a}_{i_k}\mathbf{x}_k)-\mathbf{a}_{i_k}\mathbf{a}_{j_k}^*(\mathbf{b}_{j_k}-\mathbf{a}_{j_k}\mathbf{x}_k)}{\|\mathbf{a}_{i_k}\|_2^2\|\mathbf{a}_{j_k}\|_2^2-|\mathbf{a}_{i_k}\mathbf{a}_{j_k}^*|^2},\\
\lambda_k&=\frac{\|\mathbf{a}_{i_k}\|_2^2(\mathbf{b}_{j_k}-\mathbf{a}_{j_k}\mathbf{x}_k)-\mathbf{a}_{j_k}\mathbf{a}_{i_k}^*(\mathbf{b}_{i_k}-\mathbf{a}_{i_k}\mathbf{x}_k)}{\|\mathbf{a}_{i_k}\|_2^2\|\mathbf{a}_{j_k}\|_2^2-|\mathbf{a}_{i_k}\mathbf{a}_{j_k}^*|^2}.
\end{aligned}$$
 Update $\mathbf{x}_{k+1} =\mathbf{x}_{k} + \gamma_k\mathbf{a}_{i_k}^{*} + \lambda_k\mathbf{a}_{j_k} ^{*}$.
	\end{algorithmic}
\end{algorithm}

Next, we discuss the convergence property of Algorithm \ref{algo-3-2}. The following theorem demonstrates the convergence of the expectation analysis of the algorithm.
\begin{theorem}\label{The-2}
{\rm Let $\mathbf{x}_{\star}=A^{\dagger}\mathbf{b}$ be the solution of Eq.\eqref{1-1}. Then the iterative sequence $\{\mathbf{x}_{k}\}$ generated by Algorithm \ref{algo-3-2} converges to $\mathbf{x}_{\star }$ for any initial vector $\mathbf{x}_0$ in expectation. Moreover, the corresponding error norm in expectation satisfies
}
\end{theorem}
\begin{equation}\label{3-5}
\mathbb{E} \| \overline{\mathbf{x}}_{k+1}\| _{2}^2 \leq(1-\frac{2\|A\| _F^4-2\|A\| _F^2\|A\|_2^2}{\| A\| _F^4-\|AA^*\| _F^2}\frac{\lambda_{min}(A^*A)}{\| A\| _F^2})^{k+1}\| \overline{\mathbf{x}}_0 \|_{2}^2
\end{equation}
for $k=0,1,2,\cdots.$

\begin{proof}
As seen from Algorithm \ref{algo-3-2}, we have
\begin{equation}\label{3-6}
\begin{aligned}
&\sum_{i,j=1,i\ne j}^{m}(\|\mathbf{a}_{j}\|_2^2|\mathbf{a}_{i}\overline{\mathbf{x}}_k|^2+\|\mathbf{a}_{i}\|_2^2|\mathbf{a}_{j}\overline{\mathbf{x}}_k|^2-(\mathbf{a}_{i}\overline{\mathbf{x}}_k)^* \mathbf{a}_{i}\mathbf{a}_{j}^* \mathbf{a}_{j}\overline{\mathbf{x}}_k-(\mathbf{a}_{j}\overline{\mathbf{x}}_k)^* \mathbf{a}_{j}\mathbf{a}_{i}^* \mathbf{a}_{i}\overline{\mathbf{x}}_k)\\
&=\sum_{i,j=1}^{m}(\|\mathbf{a}_{j}\|_2^2|\mathbf{a}_{i}\overline{\mathbf{x}}_k|^2+\|\mathbf{a}_{i}\|_2^2|\mathbf{a}_{j}\overline{\mathbf{x}}_k|^2-(\mathbf{a}_{i}\overline{\mathbf{x}}_k)^* \mathbf{a}_{i}\mathbf{a}_{j}^* \mathbf{a}_{j}\overline{\mathbf{x}}_k-(\mathbf{a}_{j}\overline{\mathbf{x}}_k)^* \mathbf{a}_{j}\mathbf{a}_{i}^* \mathbf{a}_{i}\overline{\mathbf{x}}_k),\\
\end{aligned}
\end{equation}
and
\begin{equation}\label{3-7}
\begin{aligned}
&\sum\limits_{i,j = 1,i\ne j}^m {\frac{|\mathbf{a}_{i}\times \mathbf{a}_{j}|^2}{\|A\|_F^4-\|A A^*\|_F^2}}=\sum\limits_{i,j = 1}^m {\frac{ |\mathbf{a}_{i}\times \mathbf{a}_{j}|^2}{\|A\|_F^4-\|A A^*\|_F^2}}.
\end{aligned}
\end{equation}
This implies that the case $i=j$ is invalid, since its probability equals $0$. Moreover, it follows from Eq.\eqref{3-4} that
\begin{equation}\label{3-8}
\begin{aligned}
&\mathbb{E}_{\rm{k}} \| \overline{\mathbf{x}}_{k+1}\|_{2}^2
= \| \overline{\mathbf{x}}_k\| _{2}^2 - {\mathbb{E}_{\rm{k}}}\parallel \mathbf{x}_{k+1} - \mathbf{x}_{k}{\parallel _{2}^2}\\
&\leq \|\overline{\mathbf{x}}_k\|_2^2-\sum\limits_{i_k,j_k = 1}^m {\frac{ |\mathbf{a}_{i_k}\times \mathbf{a}_{j_k}|^2}{\|A\| _F^4-\|A A^*\|_F^2}} \frac{\varphi_k}{\|\mathbf{a}_{i_k}\|_2^2\|\mathbf{a}_{j_k}\|_2^2-|\mathbf{a}_{i_k}\mathbf{a}_{j_k}^*|^2}\\
&=\parallel {\overline{\mathbf{x}}_k}{\parallel _{2}^2}-\frac{2\|A\| _F^2\|\mathbf{b}-A\mathbf{x}_{k}\|_2^2-2
\langle A^*
A\overline{\mathbf{x}}_k, A^* A\overline{\mathbf{x}}_k\rangle}{\| A\| _F^4-\|A A^*\|_F^2}\\
&=\parallel {\overline{\mathbf{x}}_k}{\parallel _{2}^2}-\frac{2\|A\| _F^2\|\mathbf{b}-A\mathbf{x}_{k}\|_2^2-2\|A^*(\mathbf{b}-A\mathbf{x}_{k})\|_2^2}{\| A\| _F^4-\|A A^*\|_F^2}\\
&\leq \parallel {\overline{\mathbf{x}}_k}{\parallel _{2}^2}-\frac{2\|A\| _F^2-2\|A\|_2^2}{\|A\| _F^4-\|A A^* \|_F^2}\|\mathbf{b}-A\mathbf{x}_{k}\|_2^2\\
&\leq(1-\frac{2\|A\| _F^4-2\|A\| _F^2\|A\|_2^2}{\| A\| _F^4-\|AA^*\| _F^2}\frac{\lambda_{min}(A^*A)}{\| A\| _F^2})\parallel \overline{\mathbf{x}}_k{\parallel _{2}^2},
\end{aligned}
\end{equation}
where $\lambda_{min}(A^*A)$ denotes the smallest nonzero eigenvalue of $A^*A$. 
Taking full expectations on both sides of Eq.\eqref{3-8}, it follows that
$$
\mathbb{E} \| \overline{\mathbf{x}}_{k+1}\|_{2}^2 \leq(1-\frac{2\|A\| _F^4-2\|A\| _F^2\|A\|_2^2}{\| A\| _F^4-\|AA^*\| _F^2}\frac{\lambda_{min}(A^*A)}{\| A\| _F^2}) \mathbb{E} \parallel \overline{\mathbf{x}}_k{\parallel _{2}^2}. 
$$
The theorem follows immediately by induction.
\end{proof}

\begin{theorem}\label{The-3}
{\rm Let $\mathbf{x}_{\star}=A^{\dagger}\mathbf{b}$ be the solution of Eq.\eqref{1-1}. Then the iterative sequence $\{\mathbf{x}_{k}\}$ generated by Algorithm \ref{algo-3-2} is convergent to $\mathbf{x}_{\star}$ for any initial vector $\mathbf{x}_0$. Moreover, the convergence factor generated by Theorem \ref{The-2} yields
$$1-\frac{2\|A\| _F^4-2\|A\| _F^2\|A\|_2^2}{\| A\| _F^4-\|AA^*\| _F^2}\frac{\lambda_{min}(A^*A)}{\| A\| _F^2}< 1-\frac{\lambda_{min}(A^*A)}{\| A\| _F^2},
$$
which converges faster than the RK method.
}
\end{theorem}
\begin{proof} 
Compared Theorem \ref{The-2} with Theorem $2$ presented in \cite{Strohmer}, we only need to prove that $\frac{2\|A\| _F^4-2\|A\| _F^2\|A\|_2^2}{\| A\| _F^4-\|AA^*\| _F^2}>1$ holds, i.e.,
$$
\|A\| _F^4+\|AA^*\|_F^2> 2\|A\| _F^2\|A\|_2^2,
$$
where $\|AA^*\|_{F}^{2}=\sum_{i=1}^{t}\sigma_i^4$. From Eq.\eqref{2-2}, it follows that
$$
\begin{aligned}
&\|A\| _F^4+\|AA^*|_F^2=\sum_{i=1}^{t}\sigma_i^2\sum_{i=1}^{t}\sigma_i^2+\sum_{i=1}^{t}\sigma_i^4\\
&\ge\sum_{i=1}^{t}\sigma_i^2\max\sigma_i^2+\sum_{i=1}^{t}\sigma_i^2\sum_{i=1,i\ne s}^{t}\sigma_i^2+\sum_{i=1}^{t}\sigma_i^4\\
&=\sum_{i=1}^{t}\sigma_i^2\max\sigma_i^2+(\max\sigma_i^2+\sum_{i=1,i\ne s}^{t}\sigma_i^2)\sum_{i=1,i\ne s}^{t}\sigma_i^2+\sum_{i=1}^{t}\sigma_i^4\\
&=\sum_{i=1}^{t}\sigma_i^2\max\sigma_i^2+\max\sigma_i^2\sum_{i=1,i\ne s}^{t}\sigma_i^2+\sum_{i=1,i\ne s}^{t}\sigma_i^2\sum_{i=1,i\ne s}^{t}\sigma_i^2+\max\sigma_i^4+\sum_{i=1,i\ne s}^{t}\sigma_i^4\\
&=\sum_{i=1}^{t}\sigma_i^2\max\sigma_i^2+(\max\sigma_i^2\sum_{i=1,i\ne s}^{t}\sigma_i^2+\max\sigma_i^4)+\sum_{i=1,i\ne s}^{t}\sigma_i^2\sum_{i=1,i\ne s}^{t}\sigma_i^2+\sum_{i=1,i\ne s}^{t}\sigma_i^4\\
&=2\sum_{i=1}^{t}\sigma_i^2\max\sigma_i^2+\sum_{i=1,i\ne s}^{t}\sigma_i^2\sum_{i=1,i\ne s}^{t}\sigma_i^2+\sum_{i=1,i\ne s}^{t}\sigma_i^4\\
&>2\sum_{i=1}^{t}\sigma_i^2\max\sigma_i^2=2\|A\| _F^2\|A\|_2^2.
\end{aligned}
$$
\end{proof}
Clearly, the proved inequality is true. This shows that the convergence factor of the TRK method is uniformly smaller than that of RK \cite{Strohmer} concerning the iteration index, which converges faster than the latter.

However, Algorithm \ref{algo-3-2} may still be time-consuming for big data problems, since we have to select a row index pair $(i_k, j_k)$ at each iteration. To address this issue, we perform Algorithm \ref{algo-3-2} with simple random sampling to save the storage and computational operations. As the most straightforward way of all probability sampling methods, simple random sampling \cite{Olken} only requires a single random selection without the whole advanced knowledge of the dataset. In this selection method, all the individuals have an equal probability of being selected, which ensures an unbiased,
representative, and equal probability of the dataset. Hence, we take a small portion of rows as a sample and then select a row index pair from the obtained sample. The main advantage is that we only need to compute some cross-product-like constants corresponding to the selected simple sampling set, rather than the whole dataset. This idea stems from Chebyshev's (weak) law of large numbers.
\begin{theorem}\cite{Carlton}\label{The-4}
\rm Denote by $\{z_1,\cdots,z_n,\cdots\}$, a series of random variables, by $\mathbb{E}(z_k)$, the expectation, and by $\mathbb{D}(z_k)$, the variance of $z_k$. Then, when $\frac{1}{n^2}\sum_{k=1}^{n}\mathbb{D}(z_k)$\\$\rightarrow 0$, the following relation 
\begin{equation}\label{3-9}
\lim_{n \to \infty} P\left \{ \left|\frac{1}{n}\sum _{k=1}^{n}z_k-\frac{1}{n}\sum _{k=1}^{n}\mathbb{E} (z_k) \right|<\varepsilon \right \} =1    
\end{equation}
holds for any small positive number $\varepsilon$.
\end{theorem}

As seen from Theorem \ref{The-4}, we can obtain that if the sample size is large enough, then the sample mean will be close to the population mean. Therefore, using a small part to estimate the whole is reasonable. Putting this fact together with Algorithm \ref{algo-3-2} gives the following algorithm.
\begin{algorithm}
	\caption{Two-dimensional randomized Kaczmarz method with simple random
sampling}\label{algo-3-3}
	\begin{algorithmic}[1] 
        \REQUIRE $A,  \mathbf{b},  \mathbf{x}_{0}$\\
        \ENSURE $\mathbf{x}_{k+1}$\\
        \STATE
        For $k=0, 1, 2,\cdots $, until convergence, do:\\
        \STATE
        Generate an indicator set $\Gamma_k$, i.e., choosing $l m$ rows of $A$ by using the simple random sampling, where $0<l<1$.
        \STATE
        Select ${i_k,j_k} \in \Gamma_k$ with probability $\mathbb{P}$(rows =$i_k,j_k$)=$ \frac{|\mathbf{a}_{i_k}\times \mathbf{a}_{j_k}|^2}{\sum_{i_k,j_k\in \Gamma_k}|\mathbf{a}_{i_k}\times \mathbf{a}_{j_k}|^2}  $
        \\
 \STATE
Update $\mathbf{x}_{k+1} =\mathbf{x}_{k} + \gamma_k\mathbf{a}_{i_k}^{*} + \lambda_k\mathbf{a}_{j_k} ^{*}$ as the line $3$ of Algorithm \ref{algo-3-2}.
	\end{algorithmic}
\end{algorithm}
\begin{rem}\label{Rem-2}
\rm In Algorithm \ref{algo-3-3}, simple random sampling requires a parameter $l$, and its specific choice is problem-dependent. In practice, we usually set the parameter $l$ as $0.005$ or $0.01$. Our algorithm only generates an indicator set $\Gamma_k$, which is equivalent to drawing samples from the population with equal probability through simple random sampling. Indeed, the cost of simple random sampling can be ignored and the new method can both reduce the workload and save the storage requirements significantly.
\end{rem}

In what follows, we consider the convergence of Algorithm \ref{algo-3-3}. For the sake of generality, we assume that the sample size of $\Gamma_k$ is $lm$, unless otherwise stated.

\begin{theorem}\label{The-5}
{\rm Let $\mathbf{x}_{\star}=A^{\dagger}\mathbf{b}$ be the solution of Eq.\eqref{1-1}. Then the iterative sequence $\{\mathbf{x}_{k}\}$ generated by Algorithm \ref{algo-3-3} converges to $\mathbf{x}_{\star}$ for any initial vector $\mathbf{x}_0$ in expectation. Moreover, the corresponding error norm in expectation yields
}
\end{theorem}
\begin{equation}\label{3-10}
\mathbb{E}  \|\overline{\mathbf{x}}_{k+1}\| _{2}^2
\leq(1-\frac{1-\varepsilon_k}{1+\tilde{\varepsilon}_k}\frac{2\|A\| _F^4-2\|A\|_F^2\|A\|_2^2}{\| A\| _F^4-\|AA^*\|^2_F}\frac{\lambda_{min}(A^*A)}{\| A\| _F^2})^{k+1} \|\overline{\mathbf{x}}_0\| _{2}^2
\end{equation}
for $k=0,1,2,\cdots.$
\begin{proof}
From Chebyshev’s (weak) law of large numbers, if 
$lm$ is sufficiently large and there are two scalars $0<\varepsilon_k,\tilde{\varepsilon}_k\ll1$, then it follows that
\begin{equation}\label{3-11}
\begin{aligned}
&\frac{\sum\limits_{i,j\in \Gamma_k}(\|\mathbf{a}_j\|_2^2|\mathbf{a}_{i}\overline{\mathbf{x}}_k|^2+\|\mathbf{a}_i\|_2^2|\mathbf{a}_{j}\overline{\mathbf{x}}_k|^2-(\mathbf{a}_{i}\overline{\mathbf{x}}_k)^* \mathbf{a}_i\mathbf{a}_j^* \mathbf{a}_{j}\overline{\mathbf{x}}_k-(\mathbf{a}_{j}\overline{\mathbf{x}}_k)^* \mathbf{a}_j\mathbf{a}_{i}^* \mathbf{a}_{i}\overline{\mathbf{x}}_k)}{(lm)^2-lm-1}\\
&=\frac{\sum\limits_{i,j=1}^{m}(\|\mathbf{a}_j\|_2^2|\mathbf{a}_{i}\overline{\mathbf{x}}_k|^2+\|\mathbf{a}_i\|_2^2|\mathbf{a}_{j}\overline{\mathbf{x}}_k|^2-(\mathbf{a}_{i}\overline{\mathbf{x}}_k)^* \mathbf{a}_i\mathbf{a}_j^* \mathbf{a}_{j}\overline{\mathbf{x}}_k-(\mathbf{a}_{j}\overline{\mathbf{x}}_k)^* \mathbf{a}_j\mathbf{a}_{i}^* \mathbf{a}_{i}\overline{\mathbf{x}}_k)}{m^2-m-1}(1\pm\varepsilon_k )\\
&=\frac{2\|A\| _F^2\|\mathbf{b}-Ax\|_2^2-2\|A^*(\mathbf{b}-A\mathbf{x}_{k})\|_2^2}{m^2-m-1}(1\pm\varepsilon_k ),
\end{aligned}   
\end{equation}
and
\begin{equation}\label{3-12}
\begin{aligned}
\frac{\sum\limits_{i,j\in \Gamma_k}(|\mathbf{a}_{i}\times \mathbf{a}_{j}|^2)}{(lm)^2-lm-1}&=\frac{\sum\limits_{i,j=1}^{m}(|\mathbf{a}_{i}\times \mathbf{a}_{j}|^2)}{m^2-m-1}(1\pm\tilde{\varepsilon}_k)\\
&=\frac{\| A\| _F^4-\|AA^*\| _F^2}{m^2-m-1}(1\pm\tilde{\varepsilon}_k).  
\end{aligned}
\end{equation}
Similar to Eq.\eqref{3-8}, we have
\begin{equation}\label{3-13}
\begin{aligned}
&\mathbb{E}_{\rm{k}} \| \overline{\mathbf{x}}_{k+1}\|_{2}^2
= \| \overline{\mathbf{x}}_k\| _{2}^2 - {\mathbb{E}_{\rm{k}}}\parallel \mathbf{x}_{k+1} - \mathbf{x}_{k}{\parallel _{2}^2}\\
&\leq \|\overline{\mathbf{x}}_k\|_2^2-\sum\limits_{i_k,j_k\in \Gamma_k}{\frac{|\mathbf{a}_{i_k}\times \mathbf{a}_{j_k}|^2}{\sum\limits_{i_k,j_k\in \Gamma_k}(|\mathbf{a}_{i_k}\times \mathbf{a}_{j_k}|^2)}} \frac{\varphi_k}{\|\mathbf{a}_{i_k}\|_2^2\|\mathbf{a}_{j_k}\|_2^2-|\mathbf{a}_{i_k}\mathbf{a}_{j_k}^*|^2}\\
&=\|\overline{\mathbf{x}}_k\|_2^2-\frac{\frac{1}{
(lm)^2-lm-1}\sum\limits_{i_k,j_k\in \Gamma_k}\varphi_k}{\frac{1}{
(lm)^2-lm-1}\sum\limits_{i_k,j_k\in \Gamma_k}(|\mathbf{a}_{i_k}\times \mathbf{a}_{j_k}|^2)}\\
&=\parallel {\overline{\mathbf{x}}_k}{\parallel _{2}^2}-\frac{1\pm\varepsilon_k}{1\pm\tilde{\varepsilon}_k }\frac{2\|A\| _F^2\|\mathbf{b}-A\mathbf{x}_{k}\|_2^2-2
\langle A^*
A\overline{\mathbf{x}}_k, A^* A\overline{\mathbf{x}}_k\rangle}{\| A\| _F^4-\|A A^*\|_F^2}\\
&\leq\parallel {\overline{\mathbf{x}}_k}{\parallel _{2}^2}-\frac{1-\varepsilon_k}{1+\tilde{\varepsilon}_k}\frac{2\|A\| _F^2\|\mathbf{b}-A\mathbf{x}_{k}\|_2^2-2\|A^*(\mathbf{b}-A\mathbf{x}_{k})\|_2^2}{\| A\| _F^4-\|A A^*\|_F^2}\\
&\leq \parallel {\overline{\mathbf{x}}_k}{\parallel _{2}^2}-\frac{1-\varepsilon_k}{1+\tilde{\varepsilon}_k}\frac{2\|A\| _F^2-2\|A\|_2^2}{\|A\| _F^4-\|A A^* \|_F^2}\|\mathbf{b}-A\mathbf{x}_{k}\|_2^2\\
&\leq(1-\frac{1-\varepsilon_k}{1+\tilde{\varepsilon}_k}\frac{2\|A\| _F^4-2\|A\| _F^2\|A\|_2^2}{\| A\| _F^4-\|AA^*\| _F^2}\frac{\lambda_{min}(A^*A)}{\| A\| _F^2})\parallel \overline{\mathbf{x}}_k{\parallel _{2}^2}.
\end{aligned}
\end{equation}
Again, taking the full expectation for both sides of Eq.\eqref{3-13}, it follows that
$$
\mathbb{E}\|\overline{\mathbf{x}}_{k+1}\| _{2}^2
\leq(1-\frac{1-\varepsilon_k}{1+\tilde{\varepsilon}_k}\frac{2\|A\| _F^4-2\|A\|_F^2\|A\|_2^2}{\| A\| _F^4-\|AA^*\|^2_F}\frac{\lambda_{min}(A^*A)}{\| A\| _F^2})\mathbb{E}\|\overline{\mathbf{x}}_k\| _{2}^2,
$$
holds for $k=1,2,\cdots$. Consequently,  Eq.\eqref{3-9} follows immediately by induction.
\end{proof}

\begin{rem}\label{Rem-3}
\rm The convergence factor of Algorithm \ref{algo-3-3} may be slightly larger than that of Algorithm \ref{algo-3-2}, however, when both the positive parameters $\varepsilon_k$ and $ \tilde{\varepsilon}_k$ are sufficiently less than $1$, we can still assume that $$1-\frac{1-\varepsilon_k}{1+\tilde{\varepsilon}_k}\frac{2\|A\| _F^4-2\|A\|_F^2\|A\|_2^2}{\| A\| _F^4-\|AA^*\|^2_F}\frac{\lambda_{min}(A^*A)}{\| A\| _F^2}< 1-\frac{\lambda_{min}(A^*A)}{\| A\| _F^2}.$$
This shows that Algorithm \ref{algo-3-3} is still better than that of RK. Moreover, Algorithm \ref{algo-3-3} may require more iterations than Algorithm \ref{algo-3-2}. The reason is that the indicator set $\Gamma_k$ of the former only contains the limited rows of $A$, which needs more iterations to achieve the same accuracy, whereas it consumes less time than the latter.
\end{rem}

\subsection{Two-dimensional greedy randomized Kaczmarz method}\label{subsection-3-2}
As well known that the crucial key of randomized Kaczmarz-type methods is to select the active row effectively. Observe from \cite{Bai} that the GRK solver, with the greedy selection strategy, often outperforms the RK solver in terms of the elapsed computing time. In fact, the core of the above strategy is to determine the index row corresponding to the larger entries of the residual vector $\mathbf{r}_k$. More precisely, if $|\mathbf{r}_k^{(i)}|>|\mathbf{r}_k^{(j)}|$,$i,j\in \left \{  1,2,....m \right \}$, then the $i$-th row has a larger probability of being selected than the $j$-th row, which leads to the larger elements of the residual vector to be eliminated first as much as possible. Thus we hope that two larger entries of the residual vector have a larger probability to be selected at each iteration, such that can eliminate them as much as possible. Putting this result with Algorithm \ref{algo-3-2} gives the following algorithm.

\begin{algorithm}[htbp]
	\caption{Two-dimensional greedy randomized Kaczmarz method}\label{algo-3-4}
	\begin{algorithmic}[1] 
        \REQUIRE $A,  \mathbf{b},  \mathbf{x}_{0}$\\
        \ENSURE $\mathbf{x}_{k+1}$\\
        \STATE
       For $k=0, 1, 2,\cdots $, until convergence, do:\\
        \STATE
        Compute\\
        \begin{equation}\label{3-14}
         \epsilon  _{k} = \frac{1}{2}\left(\frac{1}{{\parallel \mathbf{b} - A\mathbf{x}_{k}{\parallel _{1}}}-\varrho_k}\mathop {\max }\limits_{\substack{i_k \in [m]\\i_k\neq i_{max}}} \left\{ \frac{|\mathbf{b}_{i_k}-\mathbf{a}_{i_k}\mathbf{x}_k|}{\|\mathbf{a}_{i_k}\|_2}\right\} + \frac{1}{{\parallel A\parallel _{2,1}}-\rho_k}\right),
        \end{equation}
        where $\varrho_k=|\mathbf{b}_{i_{max}}-\mathbf{a}_{i_{max}}\mathbf{x}_k|$,        $\rho_k=\|\mathbf{a}_{i_{max}}\|_2$ and $\frac{\varrho_k}{\rho_k}=\max\limits_{1\leq i \leq m}\frac{|\mathbf{b}_i-\mathbf{a}_{i}\mathbf{x}_k|}{\|\mathbf{a}_i\|_2}.$
        \\
 \STATE
Determine the index set of positive integers\\
\begin{equation}\label{3-15}
\mathcal{U}_k = \left\{ {{i_k}\Big| {|\mathbf{b}_{i_k}-\mathbf{a}_{i_k}\mathbf{x}_k| \ge {\epsilon_k}} (\parallel \mathbf{b} - 
A{\mathbf{x}_{k}}{\parallel _{1}}-\varrho_k})\|\mathbf{a}_{i_k}\|_2 \right\}.\end{equation}\\
 \STATE
   Compute the $i$-th element of the vector $\tilde{\mathbf{r}} _{k}$ by  \\
$$\tilde{\mathbf{r} } _{k} ^{(i)}=\left\{ \begin{array}{l}
{\mathbf{b}_i} - {\mathbf{a}_i}{\mathbf{x}_{k}},\qquad  \mathrm{if} \quad i \in \mathcal{U}_{k}
\\
0.\qquad \qquad \quad \enspace \mathrm{otherwise}
\end{array}
\right.$$
    \STATE   
Select ${i_k} \in {\mathcal{U}_k}$ with probability $\mathbb{P}$(row =${i_k}$)=$ \mid {{{\tilde {\mathbf{r} }}_k}^{({i_k})}}\mid  /\left \| {{{\tilde {\mathbf{r} }}}_k} \right \|_{1} $.\\
    \STATE   
Select ${j_k} \in {\mathcal{U}_k}$ with probability $\mathbb{P}$(row =${j_k}$)=$ \mid {{{\tilde {\mathbf{r} }}_k}^{({j_k})}}\mid  /\left \| {{{\tilde {\mathbf{r} }}}_k} \right \|_{1} $.\\
            \STATE Update $\mathbf{x}_{k+1}$ as the line $3$ of Algorithm \ref{algo-3-1}.
	\end{algorithmic}
\end{algorithm}

From Algorithm \ref{algo-3-4}, we can see that the index set $\mathcal{U}_k$ is nonempty for all iteration index $k$, since
$$
\max\limits_{\substack{1\leq i_k \leq m\\i_k \neq i_{max}}}\frac{|\mathbf{b}_{i_k}-\mathbf{a}_{i_k}\mathbf{x}_k|}{\|\mathbf{a}_{i_k}\|_2}\geq \sum_{i_k=1,i_k \neq i_{max}}^{m}\frac{\|\mathbf{a}_{i_k}\|_2}{\|A\|_{2,1}-\rho_k}\frac{|\mathbf{b}_{i_k}-\mathbf{a}_{i_k}\mathbf{x}_k|}{\|\mathbf{a}_{i_k}\|_2}=\frac{\|\mathbf{b}-A\mathbf{x}_{k}\|_1-\varrho_k}{\|A\|_{2,1}-\rho_k}
$$
and
$$
\frac{|\mathbf{b}_i-\mathbf{a}_{i}\mathbf{x}_k|}{\|\mathbf{a}_i\|_2}=\max\limits_{\substack{1\leq i_k \leq m\\i_k \neq i_{max}}}\frac{|\mathbf{b}_{i_k}-\mathbf{a}_{i_k}\mathbf{x}_k|}{\|\mathbf{a}_{i_k}\|_2}\geq \frac{1}{2}\max\limits_{\substack{1\leq i_k \leq m\\i_k \neq i_{max}}}\frac{|\mathbf{b}_{i_k}-\mathbf{a}_{i_k}\mathbf{x}_k|}{\|\mathbf{a}_{i_k}\|_2}+\frac{1}{2}\frac{\|\mathbf{b}-A\mathbf{x}_{k}\|_1-\varrho_k}{\|A\|_{2,1}-\rho_k}
$$
hold, i.e., the maximum is always greater than the average. Next, we discuss the convergence of Algorithm \ref{algo-3-4}. The following theorem demonstrates the convergence of the expectation analysis of the proposed algorithm.
\begin{theorem}\label{The-6}
{\rm Let $\mathbf{x}_{\star}=A^{\dagger}\mathbf{b}$ be the solution of Eq.\eqref{1-1}. Then the iterative sequence $\{\mathbf{x}_{k}\}$ generated by Algorithm \ref{algo-3-4} converges to $\mathbf{x}_{\star}$ for any initial vector $\mathbf{x}_0$ in expectation. Moreover, the corresponding error norm in expectation satisfies
}
\end{theorem}
\begin{equation}\label{3-16}
\mathbb{E} \|\overline{\mathbf{x}}_1\|_{2}^2\leq(1-\frac{(\sqrt{\lambda_{min}(A^*A)}-\rho_k)^2}{(\|A\|_{2,1}-\rho_k)^2})\parallel \overline{\mathbf{x}}_0{\parallel _{2}^2}
\end{equation}
and
\begin{equation}\label{3-17}
\mathbb{E}_{\rm{k}} \|\overline{\mathbf{x}}_{k+1}\| _{2}^2 \leq\left\{1- \left[\frac{1}{2}\left(\frac{\sqrt{\lambda_{min}(A^*A)}-\rho_k}{\Omega-\rho_k}+\frac{\sqrt{\lambda_{min}(A^*A)}-\rho_k}{\|A\|_{2,1}-\rho_k}\right)\right]^2\right\}\|\overline{\mathbf{x}}_k\|_2^2
\end{equation}
holds for $\ k=1,2,\cdots,$ where ${\Omega}=\max \limits_{1 \leq i_0\le m}\sum\limits_{\substack{i=1\\i\ne i_{0}}}^{m}\|\mathbf{a}_i\|_2$. Moreover, we have
\begin{equation}\label{3-18}
\begin{aligned}
&\mathbb{E} \|\overline{\mathbf{x}}_{k+1}\| _{2}^2 
\leq\left\{1- \left[\frac{1}{2}\left(\frac{\sqrt{\lambda_{min}(A^*A)}-\rho_m}{\Omega-\rho_m}+\frac{\sqrt{\lambda_{min}(A^*A)}-\rho_m}{\|A\|_{2,1}-\rho_m}\right)\right]^2\right\}^k\\
&\left(1-\frac{(\sqrt{\lambda_{min}(A^*A)}-\rho_m)^2}{(\|A\|_{2,1}-\rho_m)^2}\right)\|\overline{\mathbf{x}}_0\|_2^2,
\end{aligned}
\end{equation}
in which $\rho_m=\max\limits_{i\in[m]}\|\mathbf{a}_i\|_2$.
\begin{proof}
Define
$$\mathcal{P}_k=\sum\limits_{{i_k},{j_k} \in {\mathcal{U}_k}}\frac{|\mathbf{b}_{i_k}-\mathbf{a}_{i_k}\mathbf{x}_k|}{\sum\limits_{i\in \mathcal{U}_k}|\mathbf{b}_{i}-\mathbf{a}_{i}\mathbf{x}_k|}\frac{|\mathbf{b}_{j_k}-\mathbf{a}_{j_k}\mathbf{x}_k|}{\sum\limits_{j\in \mathcal{U}_k}|\mathbf{b}_{j}-\mathbf{a}_{j}\mathbf{x}_k|}.$$ 
For $k=0$, it is easy to see that
\begin{equation}\label{3-19}
\epsilon_0(\|A\|_{2,1}-\rho_k)\ge 1 .
\end{equation}
From Eq.\eqref{3-13}, for $k=1,2,\cdots,$ it follows that
\begin{equation}\label{3-20}
\begin{aligned}
\epsilon_k(\|A\|_{2,1}-\rho_k)&=\frac{1}{2}\frac{\max \limits_{i_k \in [m],i_k\neq i_{max}} \left({{\frac{|\mathbf{b}_{i_k}-\mathbf{a}_{i_k}\mathbf{x}_k|}{\|\mathbf{a}_{i_k}\|_2}}} \right)}{\frac{1}{\|A\|_{2,1}-\rho_k}(\|\mathbf{b}-A\mathbf{x}_{k}\|_1-\varrho_k)}+\frac{1}{2}\\
&=\frac{1}{2}\frac{\max \limits_{i_k\in [m],i_k\neq i_{max}} \left({{\frac{|\mathbf{b}_{i_k}-\mathbf{a}_{i_k}\mathbf{x}_k|}{\|\mathbf{a}_{i_k}\|_2}}} \right)}{\sum\limits_{i_k=1,i_k \neq i_{max}}^{m}\frac{\|\mathbf{a}_{i_k}\|_2}{\|A\|_{2,1}-\rho_k}\frac{|\mathbf{b}_{i_k}-\mathbf{a}_{i_k}\mathbf{x}_k|}{\|\mathbf{a}_{i_k}\|_2}}+\frac{1}{2}\\
&=\frac{1}{2}\frac{\max \limits_{i_k\in [m],i_k\neq i_{max}} \left({{\frac{|\mathbf{b}_{i_k}-\mathbf{a}_{i_k}\mathbf{x}_k|}{\|\mathbf{a}_{i_k}\|_2}}} \right)}{\sum
\limits_{\substack{i_k=1\\i_k \neq i_{max}\\i_k\neq i_{k-1}}}^{m}\frac{\|\mathbf{a}_{i_k}\|_2}{\|A\|_{2,1}-\rho_k}\frac{|\mathbf{b}_{i_k}-\mathbf{a}_{i_k}\mathbf{x}_k|}{\|\mathbf{a}_{i_k}\|_2}}+\frac{1}{2}
\end{aligned}
\end{equation}
$$\begin{aligned}
&\geq\frac{1}{2}\frac{1}{\sum
\limits_{\substack{i_k=1\\i_k \neq i_{max}\\i_k\neq i_{k-1}}}^{m}\frac{\|\mathbf{a}_{i_k}\|_2}{\|A\|_{2,1}-\rho_k}}+\frac{1}{2}\\
&\geq \frac{1}{2}\left(\frac{\| A\|_{2,1}-\rho_k}{\sum
\limits_{\substack{i_k=1\\i_k \neq i_{max}\\i_k\neq i_{k-1}}}^{m}\|\mathbf{a}_{i_k}\|_2}+1\right)\geq \frac{1}{2}\left(\frac{\| A\|_{2,1}-\rho_k}{\Omega-\rho_k}+1\right).
\end{aligned}
$$
Note that the above relation $i_k\neq i_{k-1}$, since
$$
\begin{aligned}
r_{k}^{(i_{k-1})}&=\mathbf{b}_{i_{k-1}}-\mathbf{a}_{i_{k-1}}\mathbf{x}_{k}\\
&=\mathbf{b}_{i_{k-1}}-\mathbf{a}_{i_{k-1}}(\mathbf{x}_{k-1}+\gamma_{k-1}\mathbf{a}_{i_{k-1}}^{*} + \lambda_{k-1}\mathbf{a}_{j_{k-1}} ^{*})\\
&=\mathbf{b}_{i_{k-1}}-\mathbf{a}_{i_{k-1}}\mathbf{x}_{k-1}-\frac{(\|\mathbf{a}_{i_{k-1}}\|_2^2\|\mathbf{a}_{j_{k-1}}\|_2^2-|\mathbf{a}_{i_{k-1}}\mathbf{a}_{j_{k-1}}^*|^2)(\mathbf{b}_{i_{k-1}}-\mathbf{a}_{i_{k-1}}\mathbf{x}_{k-1})}{\|\mathbf{a}_{i_{k-1}}\|_2^2\|\mathbf{a}_{j_{k-1}}\|_2^2-|\mathbf{a}_{i_{k-1}}\mathbf{a}_{j_{k-1}}^*|^2}\\
&=0,
\end{aligned}
$$
similarly, the case $r_{k}^{(j_{k-1})}=0$ holds. Consequently, it follows that $\varphi _{k}=0$.

 From Eqs.\eqref{3-4}, \eqref{3-14} and \eqref{3-15}, we have
\begin{equation}\label{3-21}
\begin{aligned}
&\mathbb{E}_{\rm{k}}\|\overline{\mathbf{x}}_{k+1}\|  _{2}^2 
= \|\overline{\mathbf{x}}_k\| _{2}^2 - {\mathbb{E}_{\rm{k}}}\parallel \mathbf{x}_{k+1} - \mathbf{x}_{k}{\parallel _{2}^2}\\
&=\|\overline{\mathbf{x}}_k\|_2^2-\mathcal{P}_k\frac{
\varphi_k}{\|\mathbf{a}_{i_k}\|_2^2\|\mathbf{a}_{j_k}\|_2^2-|\mathbf{a}_{i_k}\mathbf{a}_{j_k}^*|^2}\\
&\leq\|\overline{\mathbf{x}}_k\|_2^2-\mathcal{P}_k\frac{2|\mathbf{a}_{i_k}\overline{\mathbf{x}}_k||\mathbf{a}_{j_k}\overline{\mathbf{x}}_k|(\|\mathbf{a}_{i_k}\|_2\|\mathbf{a}_{j_k}\|_2-|\mathbf{a}_{i_k}\mathbf{a}_{j_k}^*|)}{\|\mathbf{a}_{i_k}\|_2^2\|\mathbf{a}_{j_k}\|_2^2-|\mathbf{a}_{i_k}\mathbf{a}_{j_k}^*|^2}\\
&\leq\|\overline{\mathbf{x}}_k\|_2^2-\mathcal{P}_k\frac{|\mathbf{a}_{i_k}\overline{\mathbf{x}}_k||\mathbf{a}_{j_k}\overline{\mathbf{x}}_k|}{\|\mathbf{a}_{i_k}\|_2\|\mathbf{a}_{j_k}\|_2}\\
&\leq\|\overline{\mathbf{x}}_k\|_2^2-\sum\limits_{{i_k},{j_k} \in {\mathcal{U}_k}}\frac{{\epsilon_k^2}(\| \mathbf{b}- 
A{\mathbf{x}_{k}} \|_{1}-\varrho_k)^2\|\mathbf{a}_{i_k}\|_2\|\mathbf{a}_{j_k}\|_2}{(\sum\limits_{i\in \mathcal{U}_k}|\mathbf{b}_{i}-\mathbf{a}_{i}\mathbf{x}_k|)^2}\left(\frac{|\mathbf{a}_{i_k}\overline{\mathbf{x}}_k||\mathbf{a}_{j_k}\overline{\mathbf{x}}_k|}{\|\mathbf{a}_{i_k}\|_2\|\mathbf{a}_{j_k}\|_2}\right)\\
&=\|\overline{\mathbf{x}}_k\|_2^2-\epsilon_k^2(\| \mathbf{b}- 
A{\mathbf{x}_{k}} \|_{1}-\varrho_k)^2\\
&\leq\left(1-\epsilon_k^2(\sqrt{\lambda_{min}(A^*A)}-\rho_k)^2\right)\|\overline{\mathbf{x}}_k\|_2^2.
\end{aligned}
\end{equation}
Consequently, Eqs.\eqref{3-16} and \eqref{3-17} are true. Furthermore, by taking full expectations on both sides of \eqref{3-17}, it follows that
$$
\mathbb{E} \|\overline{\mathbf{x}}_{k+1}\| _{2}^2 \leq\left\{1- \left[\frac{1}{2}\left(\frac{\sqrt{\lambda_{min}(A^*A)}-\rho_k}{\Omega-\rho_k}+\frac{\sqrt{\lambda_{min}(A^*A)}-\rho_k}{\|A\|_{2,1}-\rho_k}\right)\right]^2\right\}\mathbb{E}\|\overline{\mathbf{x}}_k\|_2^2
$$
holds for $k=1,2,\cdots$. This implies that Eq.\eqref{3-18} follows immediately by induction.
\end{proof}

\subsection{Two-dimensional semi-randomized Kaczmarz methods} \label{subsection-3-3}

Note that both GRK and TGRK need to construct the indicator set $\mathcal{U}_k$ and compute probabilities during iterations, which require much more storage and operations than ones without the above formulas, especially in large data problems. To solve this problem, Jiang and Wu \cite{Jiang} proposed a semi-randomized Kaczmarz method, which outperforms the RK and GRK methods. They showed that the row corresponding to the current largest homogeneous residual is selected as the working row for higher computational efficiency. So, to improve the performance of TGRK, we propose a so-called two-dimensional ``semi-randomized" Kaczmarz (TSRK) method, built upon selecting the two active rows corresponding to the current largest and the second largest homogeneous residuals, for solving Eq.\eqref{1-1}.
\begin{algorithm}[H]
	\caption{Two-dimensional semi-randomized Kaczmarz method}
	\label{algo-3-5}
	\begin{algorithmic}[1]
        \REQUIRE $A,  \mathbf{b},  \mathbf{x}_{0}$\\
        \ENSURE $\mathbf{x}_{k+1}$\\
        \STATE
       For $k=0, 1, 2,\cdots $, until convergence, do:\\
            \STATE Select $i_k\in\{1,2,\dots,m\} $ satisfying $$\frac{|\mathbf{b}_{i_k}-\mathbf{a}_{i_k}\mathbf{x}_k|}{\|\mathbf{a}_{i_k}\|_2}=\max_{1\leq i \leq m}\frac{|\mathbf{b}_i-\mathbf{a}_{i}\mathbf{x}_k|}{\|\mathbf{a}_i\|_2}.$$
            \STATE Select $j_k\in\{1,2,\dots,m\}$ satisfying $$\frac{|\mathbf{b}_{j_k}-\mathbf{a}_{j_k}\mathbf{x}_k|}{\|\mathbf{a}_{j_k}\|_2}=\max_{1\leq j \leq m,j \neq i_k}\frac{|\mathbf{b}_j-\mathbf{a}_{j}\mathbf{x}_k|}{\|\mathbf{a}_j\|_2}.$$
            \STATE Update $\mathbf{x}_{k+1}$ as the line $3$ of Algorithm \ref{algo-3-1}.

	\end{algorithmic}  
\end{algorithm}

As seen, Algorithm \ref{algo-3-5} does not require computing probabilities or identifying index sets.
Compared with the TGRK solver, TSRK rapidly selects the two working rows at each iteration with less computational cost, which is a promising method. Then we consider the convergence of the TSRK solver. The following theorem demonstrates the convergence of the expectation analysis of the algorithm.


\begin{theorem}\label{The-7}
{\rm Let $\mathbf{x}_{\star}=A^{\dagger}\mathbf{b}$ be the solution of Eq.\eqref{1-1}. Then the iterative sequence $\{\mathbf{x}_{k}\}$ generated by Algorithm \ref{algo-3-5} converges to $\mathbf{x}_{\star}$ for any initial vector $\mathbf{x}_0$ in expectation. Moreover, the corresponding error norm in expectation satisfies
\begin{equation}\label{3-22}
\mathbb{E}_{\rm{k}} \|\overline{\mathbf{x}}_{k+1}\| _{2}^2 \leq(1-\frac{\lambda_{min}(A^*A)-\rho_k^2}{\Omega^2-\rho_k^2})\parallel \overline{\mathbf{x}}_k{\parallel _{2}^2},\qquad k=0,1,2,\cdots,
\end{equation}
where ${\Omega}=\max \limits_{1 \leq i_0\le m}\sum\limits_{\substack{i=1\\i\ne i_{0}}}^{m}\|\mathbf{a}_i\|_2$, and $\rho_k$ is defined in \eqref{3-14}}.
\end{theorem}

\begin{proof}
From Algorithm \ref{algo-3-5}, we have    
$$
\begin{aligned}
\frac{|\mathbf{a}_{i_k}\overline{\mathbf{x}}_k||\mathbf{a}_{j_k}\overline{\mathbf{x}}_k|}{\|\mathbf{a}_{i_k}\|_2\|\mathbf{a}_{j_k}\|_2}&=\max_{1\leq i \leq m}\frac{|\mathbf{b}_i-\mathbf{a}_{i}\mathbf{x}_k|}{\|\mathbf{a}_i\|_2}\max_{1\leq j \leq m, j\ne i_k}\frac{|\mathbf{b}_j-\mathbf{a}_{j}\mathbf{x}_k|}{\|\mathbf{a}_j\|_2}\\
&=\frac{\max_{1\leq i \leq m}\frac{|\mathbf{b}_i-\mathbf{a}_{i}\mathbf{x}_k|}{\|\mathbf{a}_i\|_2}\max_{1\leq j \leq m, j\ne i_k}\frac{|\mathbf{b}_j-\mathbf{a}_{j}\mathbf{x}_k|}{\|\mathbf{a}_j\|_2}}{\|r_k\|_1^2-|\mathbf{a}_{i_k}\overline{\mathbf{x}}_k|^2}(\|r_k\|_1^2-|\mathbf{a}_{i_k}\overline{\mathbf{x}}_k|^2)\\
\end{aligned}$$
$$
\begin{aligned}
&=\frac{\max_{1\leq i \leq m}\frac{|\mathbf{b}_i-\mathbf{a}_{i}\mathbf{x}_k|}{\|\mathbf{a}_i\|_2}\max_{1\leq j \leq m, j\ne i_k}\frac{|\mathbf{b}_j-\mathbf{a}_{j}\mathbf{x}_k|}{\|\mathbf{a}_j\|_2}}{\sum\limits_{\substack{i,j=1\\(i,j)\neq(i_k,i_k)}}^{m}\|\mathbf{a}_i\|_2\|\mathbf{a}_j\|_2\frac{|\mathbf{a}_{i}\overline{\mathbf{x}}_k||\mathbf{a}_{j}\overline{\mathbf{x}}_k|}{\|\mathbf{a}_{i}\|_2\|\mathbf{a}_{j}\|_2}}(\|r_k\|_1^2-|\mathbf{a}_{i_k}\overline{\mathbf{x}}_k|^2)\\
&=\frac{\max_{1\leq i \leq m}\frac{|\mathbf{b}_i-\mathbf{a}_{i}\mathbf{x}_k|}{\|\mathbf{a}_i\|_2}\max_{1\leq j \leq m, j\ne i_k}\frac{|\mathbf{b}_j-\mathbf{a}_{j}\mathbf{x}_k|}{\|\mathbf{a}_j\|_2}}{\sum\limits_{\substack{i,j=1\\i\neq i_{k-1},j\neq j_{k-1}\\(i,j)\neq(i_k,i_k)}}^{m}\|\mathbf{a}_i\|_2\|\mathbf{a}_j\|_2\frac{|\mathbf{a}_{i}\overline{\mathbf{x}}_k||\mathbf{a}_{j}\overline{\mathbf{x}}_k|}{\|\mathbf{a}_{i}\|_2\|\mathbf{a}_{j}\|_2}}(\|r_k\|_1^2-|\mathbf{a}_{i_k}\overline{\mathbf{x}}_k|^2)\\
&\geq\frac{\|r_k\|_1^2-|\mathbf{a}_{i_k}\overline{\mathbf{x}}_k|^2}{\sum\limits_{\substack{i,j=1\\i\neq i_{k-1},j\neq j_{k-1}\\(i,j)\neq(i_k,i_k)}}^{m}\|\mathbf{a}_i\|_2\|\mathbf{a}_j\|_2}\\
&\geq\frac{\lambda_{min}(A^*A)-\|\mathbf{a}_{i_k}\|_2^2}{\Omega^2-\|\mathbf{a}_{i_k}\|_2^2}\|\overline{\mathbf{x}}_k\|_{2}^2\\
&=\frac{\lambda_{min}(A^*A)-\rho_k^2}{\Omega^2-\rho_k^2}\|\overline{\mathbf{x}}_k\|_{2}^2.
\end{aligned}
$$

According to the above result and Eq.\eqref{3-4}, it follows that
$$
\begin{aligned}
\mathbb{E}_{\rm{k}}\|\overline{\mathbf{x}}_{k+1}\|  _{2}^2
&= \|\overline{\mathbf{x}}_k\| _{2}^2 - {\mathbb{E}_{\rm{k}}}\parallel \mathbf{x}_{k+1} - \mathbf{x}_{k}{\parallel _{2}^2}\\
&=\parallel \overline{\mathbf{x}}_k{\parallel _{2}^2}-\frac{
\varphi_k}{\|\mathbf{a}_{i_k}\|_2^2\|\mathbf{a}_{j_k}\|_2^2-|\mathbf{a}_{i_k}\mathbf{a}_{j_k}^*|^2}\\
&\leq\parallel \overline{\mathbf{x}}_k{\parallel _{2}^2}-\frac{2\|\mathbf{a}_{i_k}\|_2\|\mathbf{a}_{j_k}\|_2|\mathbf{a}_{i_k}\overline{\mathbf{x}}_k||\mathbf{a}_{j_k}\overline{\mathbf{x}}_k|-2|\mathbf{a}_{i_k}\overline{\mathbf{x}}_k||\mathbf{a}_{j_k}\overline{\mathbf{x}}_k||\mathbf{a}_{i_k}\mathbf{a}_{j_k}^*|}{\|\mathbf{a}_{i_k}\|_2^2\|\mathbf{a}_{j_k}\|_2^2-|\mathbf{a}_{i_k}\mathbf{a}_{j_k}^*|^2}\\
&\leq\parallel \overline{\mathbf{x}}_k{\parallel _{2}^2}-\frac{|\mathbf{a}_{i_k}\overline{\mathbf{x}}_k||\mathbf{a}_{j_k}\overline{\mathbf{x}}_k|}{\|\mathbf{a}_{i_k}\|_2\|\mathbf{a}_{j_k}\|_2}\\
&\leq(1-\frac{\lambda_{min}(A^*A)-\rho_k^2}{\Omega^2-\rho_k^2})\parallel \overline{\mathbf{x}}_k{\parallel _{2}^2}.
\end{aligned}
$$
\end{proof}

\begin{rem}\label{Rem-4}
\rm As $\frac{\sqrt{\lambda_{min}(A^*A)}+\rho_k}{\Omega+\rho_k}\geq \frac{\sqrt{\lambda_{min}(A^*A)}-\rho_k}{\Omega-\rho_k} $ and $\|A\|_{2,1}\geq\Omega$, we have
$$
\left(1-\frac{\lambda_{min}(A^*A)-\rho_k^2}{\Omega^2-\rho_k^2}\right)\leq1- \left[\frac{1}{2}\left(\frac{\sqrt{\lambda_{min}(A^*A)}-\rho_k}{\Omega-\rho_k}+\frac{\sqrt{\lambda_{min}(A^*A)}-\rho_k}{\|A\|_{2,1}-\rho_k}\right)\right]^2.
$$
Therefore, the convergence factor of the TSRK solver is less than that of the TGRK solver.
\end{rem}

Indeed, the TSRK solver does not need to compute probabilities nor construct index sets, whereas it still needs to compute the residual vectors to determine the working row. This requires the whole information of the coefficient matrix, rendering the process time-consuming. To address this issue, from Theorem \ref{The-4} and Algorithm \ref{algo-3-3}, we propose a two-dimensional semi-randomized Kaczmarz method with simple random sampling (TSRKS), which can improve the performance of Algorithm \ref{algo-3-5}.
\begin{algorithm}[H]
	\caption{Two-dimensional semi-randomized Kaczmarz method with simple random sampling}
	\label{algo-3-6}
	\begin{algorithmic}[1]
        \REQUIRE $A,  \mathbf{b},  \mathbf{x}_{0}$\\
        \ENSURE $\mathbf{x}_{k+1}$\\
        \STATE
       For $k=0, 1, 2,\cdots $, until convergence, do:\\
               \STATE
        Generate a set of indicators $\Phi_k$, i.e., choosing $\eta m$ rows of $A$ by using the simple random sampling, where $0<\eta<1$.\\
            \STATE Select $i_k\in\Phi_k $ satisfying $$\frac{|\mathbf{b}_{i_k}-\mathbf{a}_{i_k}\mathbf{x}_k|}{\|\mathbf{a}_{i_k}\|_2}=\max_{ i \in \Phi_k}\frac{|\mathbf{b}_i-\mathbf{a}_{i}\mathbf{x}_k|}{\|\mathbf{a}_i\|_2}.$$
            \STATE Select $j_k\in \Phi_k$ satisfying $$\frac{|\mathbf{b}_{j_k}-\mathbf{a}_{j_k}\mathbf{x}_k|}{\|\mathbf{a}_{j_k}\|_2}=\max_{ j \in \Phi_k,j \neq i_k}\frac{|\mathbf{b}_j-\mathbf{a}_{j}\mathbf{x}_k|}{\|\mathbf{a}_j\|_2}.$$
            \STATE Update $\mathbf{x}_{k+1}$ as the line $3$ of Algorithm \ref{algo-3-1}.

	\end{algorithmic}  
\end{algorithm}

\begin{rem}\label{Rem-5}
\rm In Algorithm \ref{algo-3-6}, simple random sampling requires a parameter $\eta$, and its specific choice is problem-dependent. For coefficient matrices of different dimensions, the choice of parameter is different. In practice, we set the parameter $\eta$ as $0.1$ or $0.01$. 
\end{rem}

In what follows, we are ready to consider the convergence of Algorithm \ref{algo-3-6}, here we assume that the sample size of $\Phi_k$ is $\eta m$. 
\begin{theorem}\label{The-8}
{\rm Let $\mathbf{x}_{\star}=A^{\dagger}\mathbf{b}$ be the solution of Eq.\eqref{1-1}. Then the iterative sequence $\{\mathbf{x}_{k}\}$ generated by Algorithm \ref{algo-3-3} converges to $\mathbf{x}_{\star}$ for any initial vector $\mathbf{x}_0$ in expectation. Moreover, for $ k=0,1,2,\cdots$,  the corresponding error norm in expectation satisfies
}
\begin{equation}\label{3-23}
\mathbb{E}_{\rm{k}}\|\overline{\mathbf{x}}_{k+1}\|  _{2} ^2\leq\left(1-(\frac{1-\varepsilon_k}{1+\tilde{\varepsilon}_k})\frac{\lambda_{min}(A^*A)-\rho_k^2}{\Omega^2-\rho_k^2}\right)\|\overline{\mathbf{x}}_k\|_2^2.
\end{equation}
\end{theorem}

\begin{proof}
From Theorem \ref{The-4}, if $m$ is large enough
and is sufficiently large, then we assume that there are two scalars $0<\varepsilon_k$, $\tilde{\varepsilon}_k\ll1$ such that
\begin{equation}\label{3-24}
\frac{\sum\limits_{\substack{i,j\in \Phi_k\\(i,j)\neq(i_k,i_k)}}|\mathbf{a}_{i}\overline{\mathbf{x}}_k||\mathbf{a}_{j}\overline{\mathbf{x}}_k|}{(\eta m)^2-1}=\frac{\|r_k\|_1^2-|\mathbf{a}_{i_k}\overline{\mathbf{x}}_k|^2}{(m-1)^2-1}(1\pm\varepsilon_k ) 
\end{equation}
and
\begin{equation}\label{3-25}
\frac{\sum\limits_{\substack{i,j\in \Phi_k\\(i,j)\neq(i_k,i_k)}}\|\mathbf{a}_i\|_2\|\mathbf{a}_j\|_2}{(\eta m)^2-1}=\frac{\sum\limits_{\substack{i,j=1\\i\neq i_{k-1},j\neq j_{k-1}\\(i,j)\neq(i_k,i_k)}}^{m}\|\mathbf{a}_i\|_2\|\mathbf{a}_j\|_2}{(m-1)^2-1}(1\pm\tilde{\varepsilon}_k ).
\end{equation}
$$
\begin{aligned}
&\frac{|\mathbf{a}_{i_k}\overline{\mathbf{x}}_k||\mathbf{a}_{j_k}\overline{\mathbf{x}}_k|}{\|\mathbf{a}_{i_k}\|_2\|\mathbf{a}_{j_k}\|_2}=\max_{i\in\Phi_k}\frac{|\mathbf{b}_i-\mathbf{a}_{i}\mathbf{x}_k|}{\|\mathbf{a}_i\|_2}\max_{j\in\Phi_k, j\ne i_k}\frac{|\mathbf{b}_j-\mathbf{a}_{j}\mathbf{x}_k|}{\|\mathbf{a}_j\|_2}\\
&=\frac{\max_{i\in\Phi_k}\frac{|\mathbf{b}_i-\mathbf{a}_{i}\mathbf{x}_k|}{\|\mathbf{a}_i\|_2}\max_{j\in\Phi_k, j\ne i_k}\frac{|\mathbf{b}_j-\mathbf{a}_{j}\mathbf{x}_k|}{\|\mathbf{a}_j\|_2}}{\sum\limits_{\substack{i,j\in \Phi_k\\(i,j)\neq(i_k,i_k)}}|\mathbf{a}_{i}\overline{\mathbf{x}}_k||\mathbf{a}_{j}\overline{\mathbf{x}}_k|}\left(\sum\limits_{\substack{i,j\in \Phi_k\\(i,j)\neq(i_k,i_k)}}|\mathbf{a}_{i}\overline{\mathbf{x}}_k||\mathbf{a}_{j}\overline{\mathbf{x}}_k|\right)\\
&=\frac{\max_{i\in\Phi_k}\frac{|\mathbf{b}_i-\mathbf{a}_{i}\mathbf{x}_k|}{\|\mathbf{a}_i\|_2}\max_{j\in\Phi_k, j\ne i_k}\frac{|\mathbf{b}_j-\mathbf{a}_{j}\mathbf{x}_k|}{\|\mathbf{a}_j\|_2}}{\sum\limits_{\substack{i,j\in \Phi_k\\(i,j)\neq(i_k,i_k)}}\|\mathbf{a}_i\|_2\|\mathbf{a}_j\|_2\frac{|\mathbf{a}_{i}\overline{\mathbf{x}}_k||\mathbf{a}_{j}\overline{\mathbf{x}}_k|}{\|\mathbf{a}_{i}\|_2\|\mathbf{a}_{j}\|_2}}\left(\sum\limits_{\substack{i,j\in \Phi_k\\(i,j)\neq(i_k,i_k)}}|\mathbf{a}_{i}\overline{\mathbf{x}}_k||\mathbf{a}_{j}\overline{\mathbf{x}}_k|\right)\\
&\geq\frac{\sum\limits_{\substack{i,j\in \Phi_k\\(i,j)\neq(i_k,i_k)}}|\mathbf{a}_{i}\overline{\mathbf{x}}_k||\mathbf{a}_{j}\overline{\mathbf{x}}_k|}{\sum\limits_{\substack{i,j\in \Phi_k\\(i,j)\neq(i_k,i_k)}}\|\mathbf{a}_i\|_2\|\mathbf{a}_j\|_2}.\\
\end{aligned}
$$
From Eq.\eqref{3-4} and the above relations, it follows that
$$
\begin{aligned}
\mathbb{E}_{\rm{k}} \|\overline{\mathbf{x}}_{k+1}\|_{2}^2 
&= \parallel \overline{\mathbf{x}}_k{\parallel _{2}^2} - {\mathbb{E}_{\rm{k}}}\parallel \mathbf{x}_{k+1} - \mathbf{x}_{k}{\parallel _{2}^2}\\
&\leq\parallel \overline{\mathbf{x}}_k{\parallel _{2}^2}-\frac{|\mathbf{a}_{i_k}\overline{\mathbf{x}}_k||\mathbf{a}_{j_k}\overline{\mathbf{x}}_k|}{\|\mathbf{a}_{i_k}\|_2\|\mathbf{a}_{j_k}\|_2}\\
&\leq\parallel \overline{\mathbf{x}}_k{\parallel _{2}^2}-\frac{\sum\limits_{\substack{i,j\in \Phi_k\\(i,j)\neq(i_k,i_k)}}|\mathbf{a}_{i}\overline{\mathbf{x}}_k||\mathbf{a}_{j}\overline{\mathbf{x}}_k|}{\sum\limits_{\substack{i,j\in \Phi_k\\(i,j)\neq(i_k,i_k)}}\|\mathbf{a}_i\|_2\|\mathbf{a}_j\|_2}\\
&=\parallel \overline{\mathbf{x}}_k{\parallel _{2}^2}-\frac{\frac{1}{(\eta m)^2-1}\sum\limits_{\substack{i,j\in \Phi_k\\(i,j)\neq(i_k,i_k)}}|\mathbf{a}_{i}\overline{\mathbf{x}}_k||\mathbf{a}_{j}\overline{\mathbf{x}}_k|}{\frac{1}{(\eta m)^2-1}\sum\limits_{\substack{i,j\in \Phi_k\\(i,j)\neq(i_k,i_k)}}\|\mathbf{a}_i\|_2\|\mathbf{a}_j\|_2}\\
&=\parallel \overline{\mathbf{x}}_k{\parallel _{2}^2}-(\frac{1\pm\varepsilon_k}{1\pm\tilde{\varepsilon}_k })\frac{\|r_k\|_1^2-|\mathbf{a}_{i_k}\overline{\mathbf{x}}_k|^2}{\sum\limits_{\substack{i,j=1\\i\neq i_{k-1},j\neq j_{k-1}\\(i,j)\neq(i_k,i_k)}}^{m}\|\mathbf{a}_i\|_2\|\mathbf{a}_j\|_2}\\
&\leq \left(1-(\frac{1-\varepsilon_k}{1+\tilde{\varepsilon}_k})\frac{\lambda_{min}(A^*A)-\rho_k^2}{\Omega^2-\rho_k^2}\right)\parallel \overline{\mathbf{x}}_k{\parallel _{2}^2}.
\end{aligned}
$$
\end{proof}
Compared Eq.\eqref{3-22} with Eq.\eqref{3-23}, one can see that the convergence factor of Algorithm \ref{algo-3-6} is slightly larger than that of Algorithm \ref{algo-3-5}. Therefore, the former may require more iterations than the latter, but consume less time in practice since the work rows of Algorithm \ref{algo-3-6} can be quickly selected by computing a small number of samples.


\section{Numerical Experiments}\label{section-4}
\label{sec:experiments}

In this section, we first report the obtained numerical results when the proposed methods and some existing Kaczmarz-type methods are performed to solve the reconstruction problem of non-uniformly sampled band-limited signals, randomly generated underdetermined and overdetermined problems, and then we illustrate the feasibility and validity of the proposed solvers compared with those methods when they are applied for image deblurring problems. Denote by ``IT", the number of iterations, and by ``CPU", the elapsed computing time in seconds. For each algorithm, we report the mean computing time in seconds and the mean number of iterations based on their average values of $5$ repeated tests. Additionally, all experiments have been carried out in Matlab
2021b on a personal computer with Inter(R) Core(TM) i5-13500HX @2.5GHz and 16.00 GB memory. Unless otherwise stated, the stopping
criteria of all tested methods are either the residual norm at the current iterate $\mathbf{x}_{k}$ satisfying 
$$\mathrm{RES}:=\left \| r_{k}  \right \| _{2}= \left \| \mathbf{b}-A\mathbf{x}_{k}  \right \| _{2}< 10^{-6}$$  or the maximum number of iterations reaching 800000. The symbol $\dagger$ indicates that the GTRK solver converges slightly slower than that of TRKS. Moreover, the initial guess for all tested
algorithms is chosen to be a zero vector. In Algorithm \ref{algo-3-3} and Algorithm \ref{algo-3-6}, we randomly select $\left \lceil l \cdot m \right \rceil$ and $\left \lceil \eta \cdot m \right \rceil$ integers from $m$ integers with equal probability to construct the indicator set $\Gamma_k=randperm(m,round(l \cdot m))$, and $ \Phi_k=randperm(m,round(\eta \cdot m))$. 
All test methods are as follows:\\
$\star$ GTRK: Generalized two-subspace randomized Kaczmarz method \cite{WWT}.\\
$\star$ TRKS: Two-dimensional randomized Kaczmarz method with simple random
sampling.\\
$\star$ GRK: Greedy randomized Kaczmarz method \cite{Bai}.\\
$\star$ TGRK: Two-dimensional greedy randomized Kaczmarz method.\\
$\star$ SRK: Semi-randomized Kaczmarz method \cite{Jiang}.\\
$\star$ TSRK: Two-dimensional semi-randomized Kaczmarz method.\\
$\star$ SRKS: Semi-randomized Kaczmarz method with simple random sampling \cite{Jiang}.\\
$\star$ TSRKS: Two-dimensional semi-randomized Kaczmarz method with simple random sampling.\\
The speed-up of TRKS, TGRK, TSRK, and TSRKS against GTRK, GRK, SRK, and SRKS is given by
$$\mathrm{speed} \text{-} \mathrm{up}=\frac{\mathrm{CPU\; of\; GTRK,GRK,SRK,or\,SRKS}}{\mathrm{CPU\; of \;TRKS,TGRK,TSRK,or \,TSRKS}}.$$ 
\begin{example}\label{Example-1}
\rm In this example, we consider the reconstruction of non-uniformly sampled bandlimited signals. Reconstructing a band-limited function $f$ from its non-uniformly spaced sampled values $\{f(t_k)\}$ is a classic problem in Fourier analysis. Here we consider trigonometric polynomials \cite{Gröchenig}, since they are applicable to the reconstruction problem of non-uniformly sampled band-limited signals. Define $f(t)=\sum_{l=-r}^{r}\mathbf{x}_le^{2 \pi ilt}$, where $\mathbf{x}=\{\mathbf{x}_l\}_{l=-r}^{r}\in \mathbb{C}^{2r+1}$. Suppose that a series of non-uniform sampling points $\{t_k\}_{k=1}^m$ and corresponding sampling values $\{f(t_k)\}_{k=1}^m$ are given, then we need to recover $f$ (or equivalently $ \mathbf{x}$) through this series of sampling values. 
Actually, the solution space of the $j$-th equation of this problem is given by the hyperplane
$$\{y:\langle y,D_r(-t_j)\rangle=f(t_j)\},$$
where $D_r(t)=\sum_{k=-r}^re^{2\pi ikt}$ denotes the Dirichlet kernel. Feichtinger and Gröchenig \cite{H.G.} believed that should consider the weighted Dirichlet kernels $\sqrt{w_j}D_r(-t_j)$ rather than $D_r(-t_j)$, where the weight $w_j=\frac{t_{j+1}-t_{j-1}}{2}, j=1,\cdots,m$. The weights should compensate for the variations in density within the sampling set.

Formulating the resulting conditions in the Fourier domain, the problem can transformed into solving the following linear equations \cite{Strohmer,Gröchenig}
\begin{equation}\label{4-1}
A\mathbf{x}=\mathbf{b},
\end{equation}
in which the entries of $A$ and $\mathbf{b}$  are given by $A_{j,k}=\sqrt{w_j}e^{2\pi ikt_j},$ and $\mathbf{b}_j=\sqrt{w_j}f(t_j)$, respectively, with $j=1,2,\dots,m, k=-r,\dots,r$. Denote $n:=2r+1$, for different $m, r$, we generate sampling points $t_j$ by extracting them randomly from a uniform distribution of $[0,1]$, and order them by magnitude. 

We performed the GTRK, TRKS, GRK, TGRK, SRK, TSRK, SRKS, and TSRKS methods to solve the Eq.\eqref{4-1}, and then recorded the obtained numerical results in Table \ref{Table-1}. As shown in this table, one can see that the number of iterations and computing time corresponding to TRKS are slightly more than GTRK, but the TGRK, TSRK, and TSRKS solvers are commonly less than that of GRK, SRK, and SRKS solvers, respectively. Specifically, the speed-up is at least $1.44$ and at most $1.66$ for TGRK compared with GRK, is at least $1.29$ and at most $1.48$ for TSRK compared with SRK, is at least $1.42$ and at most $1.71$ for TSRKS compared with SRKS. Moreover, we displayed
 the convergence curves of all methods with the case $m=3000$ and $n=301$ in Fig.\ref{Fig:1}, which coincide with the above observations.
\end{example}

\begin{table}[!htbp]
    \centering
    		\caption{Numerical results for Example \ref{Example-1} with $l=\eta=0.01$.\label{Table-1}}
    \begin{tabular}{|c|l|l|l|l|l|}
    \hline
         \multicolumn{2}{|c|}{$m\times n$} & $1000\times 101$ & $2000\times 201$ &$3000\times 301$& $4000\times 401$ \\ \hline
        \multirow{2}*{GTRK} & IT & 1260 & 2540 &4003& 5293 \\ 
        ~ & CPU &0.1062 & 0.3477 & 1.0659 & 2.5383\\ \hline

        \multirow{2}*{TRKS} & IT & 1291 & 2601 &4166& 5435 \\ 
        ~ & CPU &0.1266 & 0.3940 & 1.3370 & 3.0246\\ \hline
        \multicolumn{2}{|c|}{\textbf{speed-up}} & \textbf{$\dagger$} & \textbf{$\dagger$} & \textbf{$\dagger$} & \textbf{$\dagger$} \\ \hline
        \multirow{2}*{GRK} & IT & 611 &  1223 & 1871& 2507 \\ 
        ~ & CPU &  0.0673 & 0.2266& 0.6938 &  1.5477\\ \hline
        
        \multirow{2}*{TGRK}  & IT & 356 & 678 & 1012& 1375\\ 
        ~ & CPU &0.0468 & 0.1422 &  0.4172 &  0.9317 \\ \hline
        \multicolumn{2}{|c|}{\textbf{speed-up}} & \textbf{1.44} & \textbf{1.59} & \textbf{1.66} & \textbf{1.66} \\ \hline
        \multirow{2}*{SRK}  & IT & 516 & 1055 & 1581& 2134\\ 
        ~ & CPU &0.0552 & 0.1772 &  0.5332 &  1.3157 \\ \hline

        \multirow{2}*{TSRK}  & IT & 407 & 773 & 1110& 1436\\ 
        ~ & CPU &0.0374 & 0.1370 &  0.3899 &  0.9527 \\ \hline
        \multicolumn{2}{|c|}{\textbf{speed-up}} & \textbf{1.48} & \textbf{1.29} & \textbf{1.37} & \textbf{1.38} \\ \hline
        
        \multirow{2}*{SRKS}  & IT & 854 & 1437 & 2042& 2665\\ 
        ~ & CPU &0.0526 & 0.1767 &  0.5041 &  1.3022 \\ \hline

        \multirow{2}*{TSRKS}  & IT & 539 &829 & 1130& 1443\\ 
        ~ & CPU &\textbf{0.0370} & \textbf{0.1159} &  \textbf{0.2931} &  \textbf{0.7618} \\ \hline
        \multicolumn{2}{|c|}{\textbf{speed-up}} & \textbf{1.42} & \textbf{1.54} & \textbf{1.72} & \textbf{1.71} \\ \hline
    \end{tabular}
\end{table}

\begin{figure}[!htbp]
	\begin{center}
		\begin{minipage}[c]{0.6\textwidth}
			\includegraphics[width=2.8in]{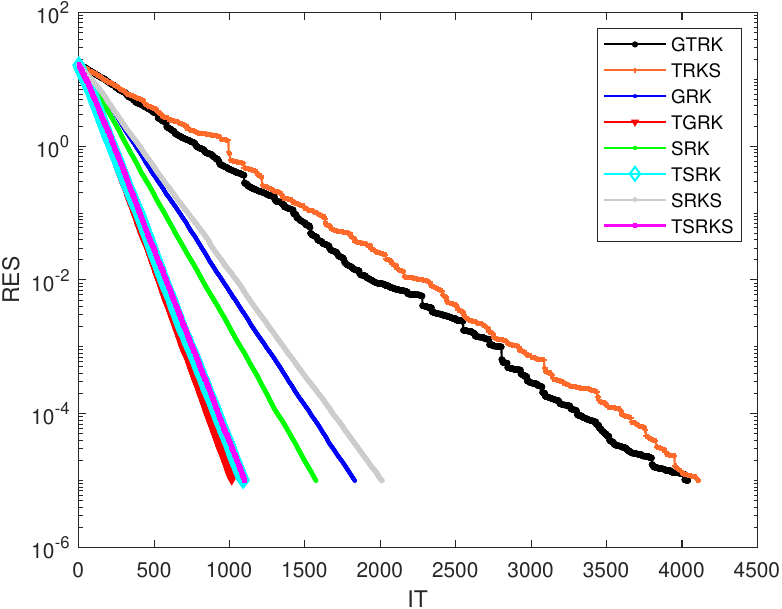}
		\end{minipage}\hspace{3.2em}
	\end{center}
	\vspace{0em}\caption{Convergence curves of Example \ref{Example-1} with the case $3000\times 301$.\label{Fig:1}}
\end{figure}

\begin{example}\label{Example-2}
\rm Consider solving Eq.\eqref{1-1} with its coefficient matrix given by $A=randn(m,n)$, where the symbol $randn(\cdot)$ denotes a MATLAB built-in function, which randomly generates some synthetic data. Moreover, we randomly generate a solution $\mathbf{x}_\star$ by the MATLAB function $randn$, with $\mathbf{x}_\star=randn(n,1)$, such that the right-hand side vector $\mathbf{b}$ is set to be $A\mathbf{x}_\star$.
The stopping criterion for the overdetermined cases is of the form
$$\mathrm{RES}:=  \frac{\| \mathbf{x}_{\star}-\mathbf{x}_{k} \| _{2}^2}{\|\mathbf{x}_{k}\|_2^2}< 10^{-6}.$$

\begin{table}[!htbp]
    \centering
    		\caption{Numerical results for Example \ref{Example-2} with $l=\eta=0.005$.\label{Table-2}}
    \begin{tabular}{|c|l|l|l|l|l|}
    \hline
        \multicolumn{2}{|c|}{$m\times n$} & $1000\times 200$ & $4000\times 600$ &$6000\times 800$& $10000\times 1000$ \\ \hline

        \multirow{2}*{GTRK} & IT & 1524 & 4219&5404&6371 \\ 
        ~ & CPU &0.0347   & 0.3027& 0.4493&0.7291\\ \hline

        \multirow{2}*{TRKS} & IT & 1550  &4164& 5359&6493 \\ 
        ~ & CPU &0.0383 & 0.2457&0.4032&0.7899\\ \hline
        \multicolumn{2}{|c|}{\textbf{speed-up}} & \textbf{$\dagger$} & \textbf{1.23} & \textbf{1.11} & \textbf{$\dagger$} \\ \hline
        \multirow{2}*{GRK} & IT & 455  & 1141&1427&1540 \\ 
        ~ & CPU &  0.0406 &  0.3857&1.1507&2.9888\\ \hline
        
        \multirow{2}*{TGRK}  & IT & 266 & 655&817&896\\ 
        ~ & CPU &0.0214   &  0.2339&0.6623&1.7602 \\ \hline
        \multicolumn{2}{|c|}{\textbf{speed-up}} & \textbf{1.90} & \textbf{1.65} & \textbf{1.74} & \textbf{1.70} \\ \hline
        \multirow{2}*{SRK}  & IT & 414 & 1037&1325&1407\\ 
        ~ & CPU &0.0255   &  0.2313 &0.8779&2.5864\\ \hline

        \multirow{2}*{TSRK}  & IT & 220  & 532&677&708\\ 
        ~ & CPU &\textbf{0.0109}  &  \textbf{0.1269}&0.4690&1.3611 \\ \hline
        \multicolumn{2}{|c|}{\textbf{speed-up}} & \textbf{2.34} & \textbf{1.82} & \textbf{1.87} & \textbf{1.90} \\ \hline
        
        \multirow{2}*{SRKS}  & IT & 1061  & 1707&1964&2044\\ 
        ~ & CPU &0.0217 &  0.2140&0.5831&1.3653 \\ \hline

        \multirow{2}*{TSRKS}  & IT & 737  & 1015&1142&1160\\ 
        ~ & CPU &0.0160   &  0.1311 &\textbf{0.3475}&\textbf{0.7798}\\ \hline
        \multicolumn{2}{|c|}{\textbf{speed-up}} & \textbf{1.34} & \textbf{1.63} & \textbf{1.68} & \textbf{1.75} \\ \hline
    \end{tabular}
\end{table}

\begin{table}[!htbp]
    \centering
    		\caption{Numerical results for Example \ref{Example-2}  with  $l=0.0001$, $\eta=0.001$.\label{Table-3}}
    \begin{tabular}{|c|l|l|l|l|l|}
    \hline
        \multicolumn{2}{|c|}{$m\times n$} & $200000\times 50$ & $200000\times 200$ &$200000\times 1000$& $200000\times 2000$\\ \hline
        \multirow{2}*{GTRK} & IT & 292 & 1148& 5744&11733 \\ 
        ~ & CPU &0.3617 & 1.3427& 7.2773&15.8114\\ \hline

        \multirow{2}*{TRKS} & IT & 274 & 1160& 5765&11638 \\ 
        ~ & CPU &0.0123  & 0.1766& 3.7783&14.2632\\ \hline
        \multicolumn{2}{|c|}{\textbf{speed-up}} & \textbf{29.40} & \textbf{7.60} & \textbf{1.93} & \textbf{1.11} \\ \hline
        
        \multirow{2}*{GRK} & IT & 45 & 173& 887&1875 \\ 
        ~ & CPU &  0.3233 &  1.8114&  28.5966&114.0079\\ \hline
        
        \multirow{2}*{TGRK}  & IT & 35 & 128& 606&1257 \\ 
        ~ & CPU &0.2174  &  1.2502& 20.7140& 79.3002\\ \hline
        \multicolumn{2}{|c|}{\textbf{speed-up}} & \textbf{1.49} & \textbf{1.45} & \textbf{1.38} & \textbf{1.44} \\ \hline
        \multirow{2}*{SRK} & IT & 26 & 116 &700& 1534 \\ 
        ~ & CPU &0.0929 & 0.8737 & 20.9293 &  92.8542\\ \hline
        
        \multirow{2}*{TSRK} & IT &16 &  61 &346& 769 \\ 
        ~ & CPU &  0.0544 & 0.4512& 10.2582 &   46.1732\\ \hline
        \multicolumn{2}{|c|}{\textbf{speed-up}} & \textbf{1.71} & \textbf{1.94} & \textbf{2.04} & \textbf{2.01} \\ \hline        
        \multirow{2}*{SRKS}  & IT & 62 & 261 & 1310& 2631 \\ 
        ~ & CPU &0.0128 & 0.1728 &  3.9808 &   18.6871\\ \hline
        
        \multirow{2}*{TSRKS}  & IT & 38 & 146 & 726& 1451 \\ 
        ~ & CPU &\textbf{0.0077} & \textbf{0.0954} &  \textbf{2.2312} & \textbf{10.3733}\\ \hline
        \multicolumn{2}{|c|}{\textbf{speed-up}} & \textbf{1.66} & \textbf{1.81} & \textbf{1.78} & \textbf{1.80}\\ \hline

    \end{tabular}
\end{table}

\begin{table}[!htbp]
    \centering
    		\caption{Numerical results for Example \ref{Example-2} with $l,\eta=0.1$.\label{Table-4}}
    \begin{tabular}{|c|l|l|l|l|l|}
    \hline
        \multicolumn{2}{|c|}{$m\times n$} & $100\times 1000$  &$300\times 5000$& $400\times 6000$& $500\times 7000$ \\ \hline
        \multirow{2}*{GTRK} & IT & 1772 & 4993& 6950&9115 \\ 
        ~ & CPU &0.1131 & 1.4865& 2.5586&4.6173\\ \hline

        \multirow{2}*{TRKS} & IT & 1727 & 5187& 6814&8829 \\ 
        ~ & CPU &0.1174  & 1.5703& 2.7450&5.0334\\ \hline
        \multicolumn{2}{|c|}{\textbf{speed-up}} & \textbf{$\dagger$} & \textbf{$\dagger$} & \textbf{$\dagger$} & \textbf{$\dagger$} \\ \hline
        \multirow{2}*{GRK} & IT & 896 & 2413& 3292&4285 \\ 
        ~ & CPU &  0.0510 &  0.7120& 1.2701&2.3085\\ \hline
        
        \multirow{2}*{TGRK}  & IT & 480 & 1223& 1681&2223 \\ 
        ~ & CPU &0.0320  &  0.3899& 0.6773& 1.2674\\ \hline
        \multicolumn{2}{|c|}{\textbf{speed-up}} & \textbf{1.59} & \textbf{1.83} & \textbf{1.88} & \textbf{1.82}\\ \hline

        \multirow{2}*{SRK}  & IT & 930  & 2314& 3264&4236\\ 
        ~ & CPU &0.0486  &  0.6558 & 1.2141&2.1946\\ \hline

        \multirow{2}*{TSRK}  & IT & 464  & 1176& 1647&2121\\ 
        ~ & CPU &\textbf{0.0230}  & \textbf{0.3588}& \textbf{0.6469}&\textbf{1.1472}\\ \hline
        \multicolumn{2}{|c|}{\textbf{speed-up}} & \textbf{2.11} & \textbf{1.83} & \textbf{1.88} & \textbf{1.91}\\ \hline
        
        \multirow{2}*{SRKS}  & IT & 949  & 2386& 3254&4307\\ 
        ~ & CPU &0.0520   &  0.6108 & 1.0841&2.0190\\ \hline

        \multirow{2}*{TSRKS}  & IT & 497  & 1225& 1635&2201\\ 
        ~ & CPU &0.0260  &  0.3720 & 0.6491&1.2185\\ \hline
        \multicolumn{2}{|c|}{\textbf{speed-up}} & \textbf{2.00} & \textbf{1.64} & \textbf{1.67} & \textbf{1.66}\\ \hline

    \end{tabular}
\end{table}

\begin{figure}[htbp]
\begin{center}
	\begin{minipage}[c]{1\textwidth}
		\includegraphics[width=2.49in]{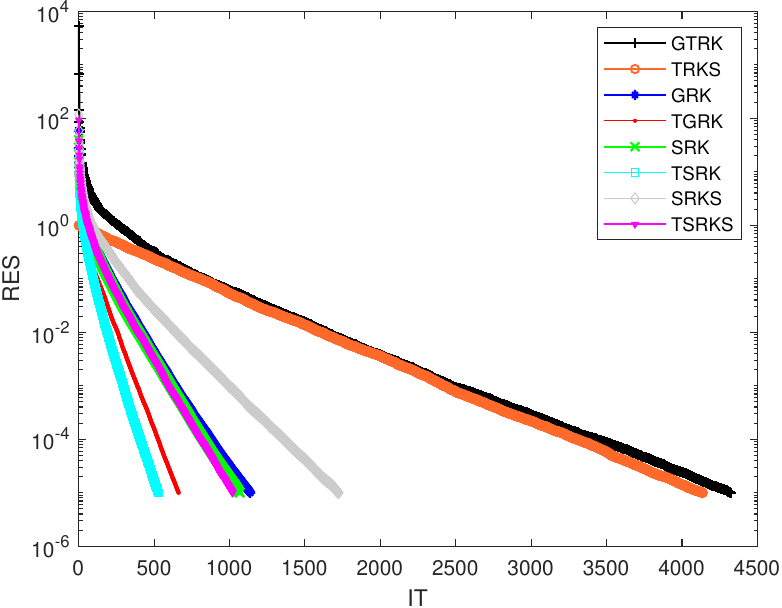} 
		\includegraphics[width=2.49in]{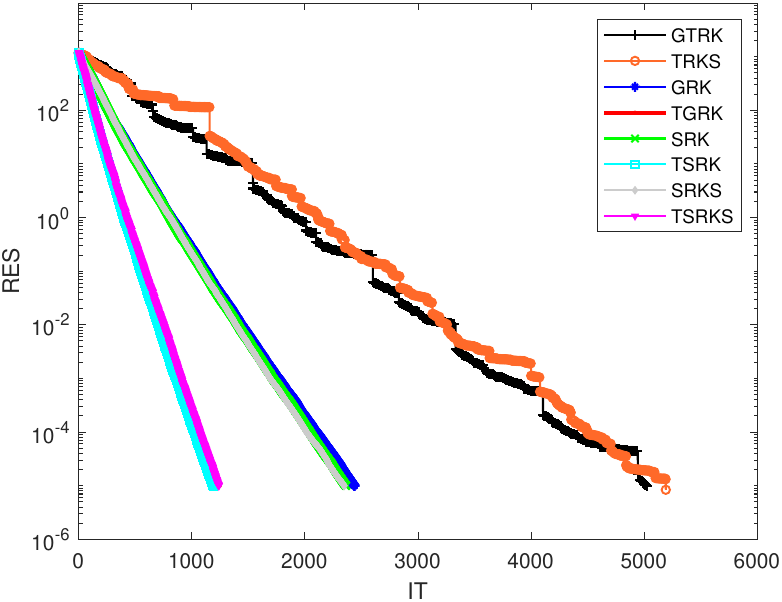}
	\end{minipage}
	\caption{ Convergence curves of Example \ref{Example-2} with the cases $4000\times 600$ (left) and $300\times 5000$ (right).}\label{Fig:2}
 \end{center}
\end{figure}

We implemented all the test methods to solve Eq.\eqref{1-1} and then listed the obtained numerical results in Tables \ref{Table-2}-\ref{Table-4}. One can see from these tables that, for all cases, TGRK, TSRK, and TSRKS solvers perform much better than GRK, SRK, and SRKS solvers in terms of the number of iterations and CPU time with significant
speed-ups, respectively. Note that, in the overdetermined case, the speed-up of TGRK against GRK is at least $1.38$ with $200000\times 1000$, and at most attaining $1.90$ with $1000\times 200$. The speed-up of TSRK against SRK is at least $1.71$ with $200000\times 50$, and at most attaining even $2.34$ with $1000\times 200$. The speed-up of TSRKS against SRKS is at least $1.34$ with $1000\times 200$, and at most attaining $1.81$ with $200000\times 200$. 
In the underdetermined case, the speed-up is at least $1.59$ for the TGRK and GRK solvers with $100\times 1000$, and at most $2.11$ for 
the TSRK and SRK solvers with $100\times 1000$. 

From the above results, we can draw two statements. One is that the two-dimensional Kaczamrz-type methods perform more competitively than traditional one-dimensional methods under the same strategy, since two working rows possess more valid information at each iteration. Moreover, we can also see that the computing time required by TSRK and TSRKS solvers is much less than that of other solvers, which perform best among other solvers. The reason is that, at each iteration, both of them only select the two active rows corresponding to the current largest and the second largest homogeneous residuals, and the latter requires a small part of the rows of the coefficient matrix, which saves computational operations and storage. The convergence curves of all methods with $4000\times 600$ and $300\times 5000$ were displayed in Fig.\ref{Fig:2}. From this figure, we can intuitively obtain that the proposed two-dimensional Kaczamrz-type methods converge faster than the one-dimensional ones.
\end{example}

\begin{small}
\begin{table}[!htbp]
    \centering
    \caption{Numerical results  of Example \ref{Example-3} with  $l=0.01,\eta=0.1$  \label{Table-5}}
    \setlength{\tabcolsep}{0.7mm}
    \begin{tabular}{|l|c|c|c|c|c|c|}
    \hline
        \multicolumn{2}{|c|}{Matrix} & D\_7 & cage8 &df2177 & mk11-b2&photogrammetry2 \\ \hline
        \multicolumn{2}{|c|}{$ m \times n$} & $1270\times971$ & $1015\times1015$ & $630\times10358$ & $6930\times990$ & $4472\times936$\\ \hline
        \multicolumn{2}{|c|}{Density}  & 1.03\%& 1.07\% & 0.34\% & 3.03\% & 0.89\% \\ \hline
        \multicolumn{2}{|c|}{Cond$\left ( A \right ) $ } & Inf& 11.41 & 2.01 & 2.39e+15 & 1.34e+08 \\ \hline
        GTRK & IT &  690414  & 224305 & 8935 &         15672  & 48984\\ 
        ~ & CPU & 50.8282 & 11.4302 & 1.6967&1.6078& 4.9297\\ \hline
        TRKS & IT &  615519 & 176353 & 8889 & 15818  & 50234\\ 
        ~ & CPU & 40.6518& 8.5208&  1.5537&2.1016& 4.3723\\ \hline
        \multicolumn{2}{|c|}{\textbf{speed-up}} & \textbf{1.25} & \textbf{1.34} & \textbf{1.10} & \textbf{$\dagger$} &\textbf{1.13}\\ \hline
        GRK & IT &  90047 & 7887
 & 4290 &  4091 & 10164\\ 
        ~ & CPU & 6.6980 & 0.4006 &  0.9212& 0.5619& 1.1015\\ \hline

        TGRK & IT & 47357 & 3981 &2170 & 2375&  5462 \\ 
        ~ & CPU & 3.7472 & 0.2303 & 0.4190 & 0.3515&  0.7162\\ \hline
        \multicolumn{2}{|c|}{\textbf{speed-up}} & \textbf{1.79} & \textbf{1.74} & \textbf{2.20} & \textbf{1.60} &\textbf{1.54}\\ \hline

        \multirow{2}*{SRK} & IT &86378 & 8254 &3785 &  3654&  9710\\ 
        ~ & CPU & 5.6188 & 0.3772 & 0.7814 &   0.2218 & 0.6253\\ \hline

        \multirow{2}*{TSRK} & IT &44277  & 4083 &1891 & 1844 &  6166\\ 
        ~ & CPU & \textbf{3.0664} & \textbf{0.2012} & \textbf{0.3494} & \textbf{0.1356}& 0.4410\\ \hline
        \multicolumn{2}{|c|}{\textbf{speed-up}} & \textbf{1.83} & \textbf{1.87} & \textbf{2.24} & \textbf{1.64} &\textbf{1.42}\\ \hline
        \multirow{2}*{SRKS} & IT &  89686 & 8329 &3793 & 3756 & 10065\\ 
        ~ & CPU & 5.2169 & 0.3456 & 0.5992 & 0.2493& 0.6379\\ \hline

        \multirow{2}*{TSRKS} & IT & 50012 & 4225 &1913 &1931 & 5089\\ 
        ~ & CPU &3.6592 & 0.2107& 0.3605 & 0.1467& \textbf{0.3848}\\ \hline
        \multicolumn{2}{|c|}{\textbf{speed-up}} & \textbf{1.43} & \textbf{1.64} & \textbf{1.66} & \textbf{1.70} &\textbf{1.66}\\ \hline
    \end{tabular}
\end{table}
\end{small}
\begin{figure}[htbp]
\begin{center}
	\begin{minipage}[c]{0.95\textwidth}
		\includegraphics[width=2.4in]{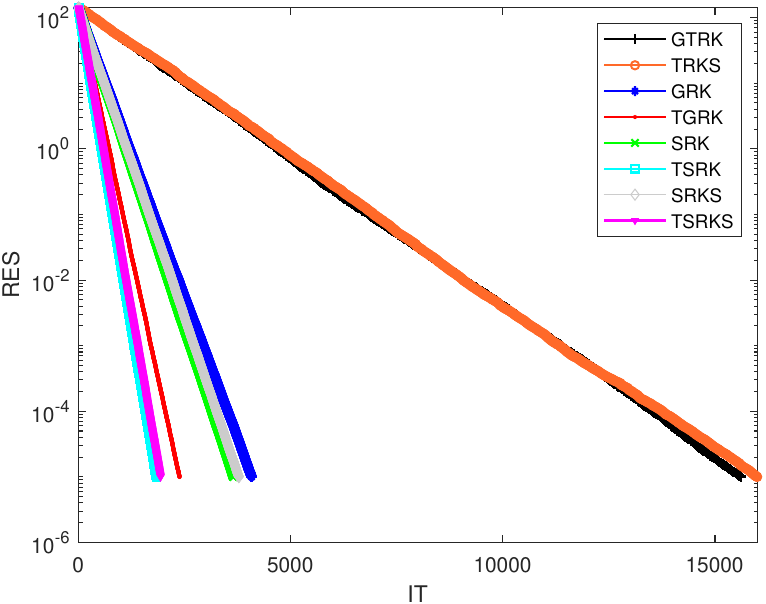}\
		\includegraphics[width=2.44in]{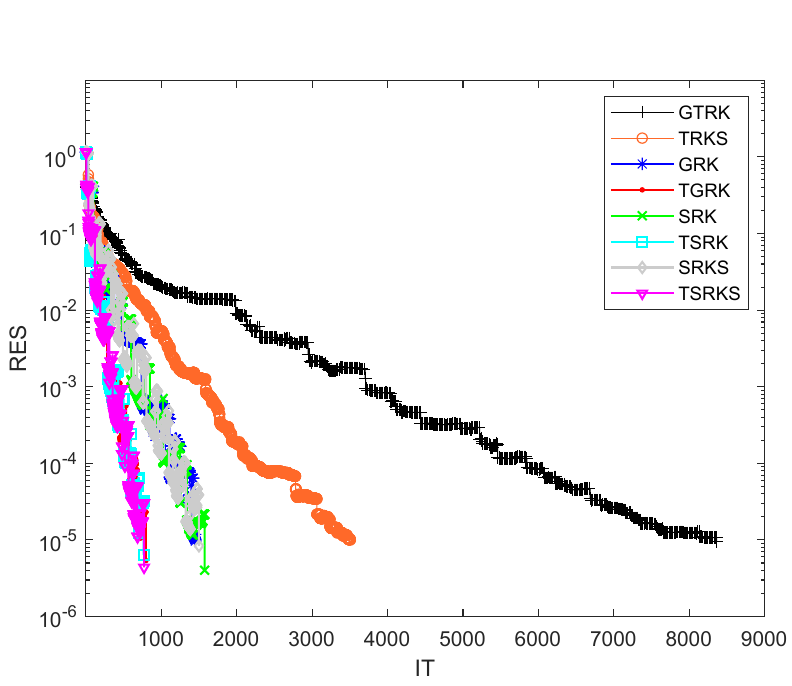}
	\end{minipage}
	\caption{Convergence curves of Example \ref{Example-3} with the cases mk11-b2 (left) and pivtol (right).}\label{Fig:3}
 \end{center}
\end{figure}
\begin{figure}[htbp]
\begin{center}
	\begin{minipage}[c]{0.95\textwidth}
		\includegraphics[width=2.4in]{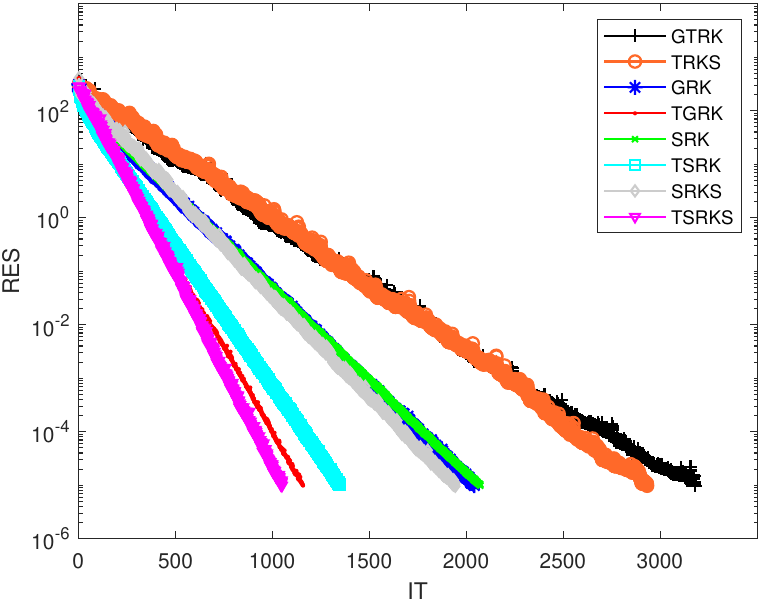}\
		\includegraphics[width=2.4in]{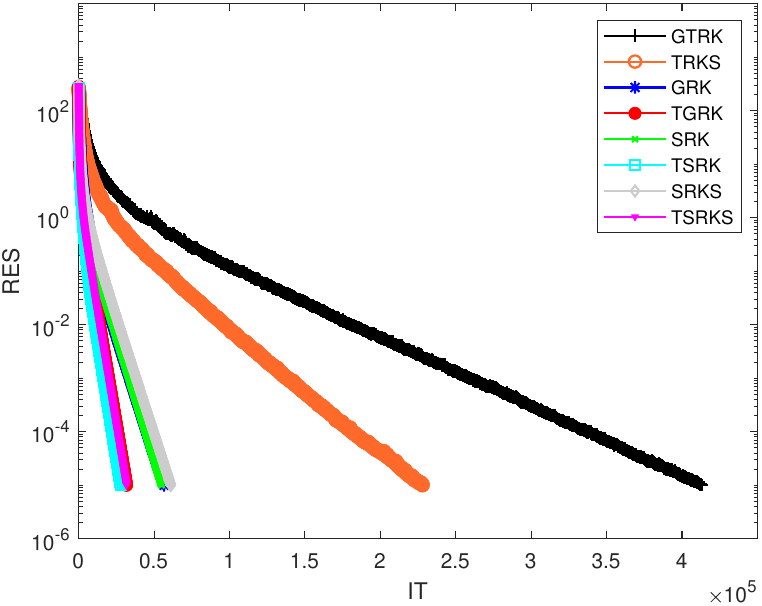}
	\end{minipage}
	\caption{Convergence curves of Example \ref{Example-3} with the cases bibd\_15\_7 (left) and relat7 (right).}\label{Fig:4}
 \end{center}
\end{figure}
\begin{example}\label{Example-3}
\rm Consider solving Eq.\eqref{1-1} with its coefficient matrix coming from the University of Florida Sparse Matrix Collection \cite{Davis}. The chosen matrices are mainly used for the least-squares, statistical mathematical, combinatorial, and linear programming problems, which possess certain structures, such as thin ($m>n$) (e.g., D\_7, mk11-b2 ), fat ($m<n$) (e.g., cari, nemsafm), or square ($m=n$) (e.g., cage8, pivtol). The right-hand side vectors are generated similarly in Example \ref{Example-2}. 

In Tables \ref{Table-5}-\ref{Table-7}, we reported the numerical results regarding the iterations and computing time. As shown in these tables, a rough conclusion is that the TGRK, TSRK, and TSRKS solvers perform much better than the GRK, SRK, and SRKS solvers, respectively, and the speed-up of TRKS is generally improved in comparison with GTRK, especially in Franz10 reached $30.96$. More precisely, the speed-up of TGRK compared with GRK is at least $1.26$ with the case relat7, at most attaining $2.20$ with the case df2177, TSRK compared with SRK is at least $1.25$ with the case shar\_te2-b1, at most attaining $2.24$ with the case df2177, TSRKS compared with SRKS is at least $1.33$ with the case abtaha1, at most attaining $1.80$ with the case cari.
This reflects that those methods with the simple random sampling strategy are more favorable for big data problems, and the proposed two-dimensional Kaczmarz-type methods, based on two carefully selected working rows at each iteration, are more competitive than those of one-dimensional ones. Additionally, we depicted the convergence curves of all tested methods with the cases mk11-b2 and pivtol in Fig.\ref{Fig:3}, the cases bibd\_15\_7 and relat7 in Fig.\ref{Fig:4}. As seen, the TGRK, TSRK, and TSRKS solvers still work more effectively than the GRK, SRK, and SRKS solvers, respectively, which coincides with the abovementioned conclusion.
\end{example}

\begin{table}[!htbp]
    \centering
    \caption{Numerical results  of Example \ref{Example-3}, $l,\eta=0.1$ \label{Table-6}}
    \begin{tabular}{|l|c|c|c|c|c|c|}
    \hline
        \multicolumn{2}{|c|}{Matrix} & cari & pivtol &bibd\_15\_7 & nemsafm &rosen1 \\ \hline
        \multicolumn{2}{|c|}{$ m \times n$} & $400\times1200$ & $102\times102$ & $105\times6435$  & $334\times2348$ & $520\times1544$\\ \hline
        \multicolumn{2}{|c|}{Density}  & 31.83\% & 2.94\% & 20\% &  0.36\% & 2.96\% \\ \hline
        \multicolumn{2}{|c|}{Cond$\left ( A \right ) $ } & 3.13& 109.61 & 7.65 &  4.77 & 58.48 \\ \hline
        GTRK & IT &  4770 & 8349& 3175 & 18486 &393175 \\
        ~ & CPU &  0.5889 & 0.2834 & 0.9026 & 1.0455&  29.3738\\ \hline
        TRKS & IT &  4629& 3314& 2972 &  16290&374092 \\ 
        ~ & CPU & 0.6023 & 0.1037 & 0.7843& 0.9245&  33.2062\\ \hline
        \multicolumn{2}{|c|}{\textbf{speed-up}} & \textbf{$\dagger$} & \textbf{2.73} & \textbf{1.15} & \textbf{1.13} &\textbf{$\dagger$}\\ \hline
        GRK & IT &  2134 & 1490 & 2040  &  2777 &22188 \\ 
        ~ & CPU & 0.2389 & 0.0350 & 0.5775&0.1534& 1.3104\\ \hline
        
        TGRK & IT & 1046 &  820 &  1146&  1429& 11360\\ 
        ~ & CPU & 0.1353 & 0.0247 & 0.3315 & 0.0844& 0.7519\\ \hline
        \multicolumn{2}{|c|}{\textbf{speed-up}} & \textbf{1.77} & \textbf{1.42} & \textbf{1.75} & \textbf{1.60} &\textbf{1.74}\\ \hline

        \multirow{2}*{SRK} & IT & 2277 & 1532&2117 & 2996& 21477\\ 
        ~ & CPU & 0.2370 & 0.0344 & 0.5889 &  0.1517& 1.1872\\ \hline

        TSRK & IT & 1121 & 744& 1350 & 1469 & 10809\\ 
        ~ & CPU &  0.1360 & \textbf{0.0216} & \textbf{ 0.3805} & \textbf{0.0785}& 0.6812\\ \hline
       \multicolumn{2}{|c|}{\textbf{speed-up}} & \textbf{1.74} & \textbf{1.59} & \textbf{1.55} & \textbf{1.93} &\textbf{1.74}\\ \hline
        SRKS & IT & 2199 & 1542 & 1955 & 2985 & 21928\\ 
        ~ & CPU & 0.2254 &0.0320 & 0.4300 & 0.1294& 1.1037\\ \hline

        TSRKS & IT & 1047 & 782 & 1055 & 1487& 10376\\ 
        ~ & CPU & \textbf{0.1250} & 0.0225 & 0.2918& 0.0797& \textbf{0.6362}\\ \hline
        \multicolumn{2}{|c|}{\textbf{speed-up}} & \textbf{1.80} & \textbf{1.42} & \textbf{1.47} & \textbf{1.62} &\textbf{1.73}\\ \hline
    \end{tabular}
\end{table}

\begin{table}[!htbp]
    \centering
    \caption{Numerical results  of Example \ref{Example-3} with $l=0.001,\eta=0.01$ \label{Table-7}}
    \setlength{\tabcolsep}{0.7mm}
    \begin{tabular}{|l|c|c|c|c|c|c|}
    \hline
        \multicolumn{2}{|c|}{Matrix} & shar\_te2-b1 & abtaha1 &abtaha2 & relat7 &Franz10 \\ \hline
        \multicolumn{2}{|c|}{$ m \times n$} & $17160\times286$ & $14596\times209$ & $37932\times331$ & $21924\times1045$ & $19588\times4164$\\ \hline
        \multicolumn{2}{|c|}{Density}  & 0.70\% & 1.68\% & 1.09\% & 0.36\% & 0.12\% \\ \hline
        \multicolumn{2}{|c|}{Cond$\left ( A \right ) $ } & 1.56e+13& 12.23 & 12.22 & Inf & 1.27e+16 \\ \hline
        GTRK & IT &  4966 &  122429 & 210653
 & 417898 &176833 \\ 
        ~ & CPU &  0.8564 & 22.1822 &  108.0814&402.3947& 614.1634\\ \hline
        TRKS & IT &  4702 &  120614 & 209497
 & 232083 &90414 \\ 
        ~ & CPU &   0.7364 & 16.4418 & 130.2689&45.7278&  19.8401\\ \hline
        \multicolumn{2}{|c|}{\textbf{speed-up}} & \textbf{1.16} & \textbf{1.35} & \textbf{$\dagger$} & \textbf{8.80} &\textbf{30.96}\\ \hline
        GRK & IT &  1090 & 2388 & 2805 & 57840 &22188 \\ 
        ~ & CPU &  0.3462 & 0.5726 &   1.6882&19.9134& 13.1041\\ \hline
        
        TGRK & IT & 708 &  1607 &1898 & 31516& 11360\\ 
        ~ & CPU &  0.1812 &  0.4108 & 1.2024 &15.8139& 7.5192\\ \hline
        \multicolumn{2}{|c|}{\textbf{speed-up}} & \textbf{1.91} & \textbf{1.39} & \textbf{1.40} & \textbf{1.26} &\textbf{1.74}\\ \hline

        \multirow{2}*{SRK} & IT &  891 & 1807 &1915 & 54150& 24136\\ 
        ~ & CPU & 0.1354 & 0.2786 & 0.6996 & 11.5236& 6.2769\\ \hline

        TSRK & IT &  746 & 970 &987 & 28677& 12837\\ 
        ~ & CPU & 0.1054 &  \textbf{0.1685} &  \textbf{0.4050} & 7.1957& 3.2327\\ \hline
        \multicolumn{2}{|c|}{\textbf{speed-up}} & \textbf{1.25} & \textbf{1.65} & \textbf{1.73} & \textbf{1.60} &\textbf{1.94}\\ \hline
        SRKS & IT & 1187 & 4764 &3927 &59966& 24854\\ 
        ~ & CPU & 0.1080 & 0.4737 & 1.1283 &10.3533& 4.9984\\ \hline

        TSRKS & IT & 650 & 2864 & 2363 & 29045&  12858\\ 
        ~ & CPU &  \textbf{0.0754} & 0.3572 & 0.6804 &  \textbf{6.0947}& \textbf{2.8876}\\ \hline
        \multicolumn{2}{|c|}{\textbf{speed-up}} & \textbf{1.43} & \textbf{1.33} & \textbf{1.66} & \textbf{1.70} &\textbf{1.73}\\ \hline
        
    \end{tabular}
\end{table}

To further demonstrate the effectiveness of the proposed methods, we ran all tested methods to solve Eq.\eqref{1-1} in image deblurring problems, and then
compared them numerically.  Denote by $X\in\mathbb{R}^{n\times n}$, a real matrix of order $n$ that can be regarded as a grayscale image image, 
and by $\mathbf{x}=\mbox{vec}(X)\in\mathbb{R}^{n^2}$, a vector generated by stacking its columns. We now assume that there exists a blurring matrix $A$, which models various blurs, e.g., shake blur and Gaussian blur, corrupting the image $\mathbf{x}$ so that $\mathbf{b}$ is the observed image. Then the blurring image $X$ will be recovered if Eq.\eqref{1-1} is solved quickly. The performance of all solvers is evaluated by the peak signal-to-noise ratio ({PSNR}) in decibels ({dB}), structural similarity ({SSIM}), and error norm ({EN}). Define
$$\begin{aligned}
\mbox{PSNR}(\mathbf{x}_k)&=10\log_{10}(\frac{mnd^2}{\|\mathbf{x}-\mathbf{x}_k\|_2^2}),\quad \mbox{EN}(\mathbf{x}_k)=\frac{\|\mathbf{x}-\mathbf{x}_k\|_2}{\|\mathbf{x}\|_2},  \\
\mbox{SSIM}(\mathbf{x}_k)&=\frac{(2\mu_{\mathbf{x}}\mu_{\mathbf{x}_k}+c_1)(2\sigma_{\mathbf{x}\mathbf{x}_k}+c_2)}{(\mu_{\mathbf{x}}^2+\mu_{\mathbf{x}_k}^2+c_1)(\sigma_{\mathbf{x}}^2+\sigma_{\mathbf{x}_k}^2+c_2)},
\end{aligned}
$$
where $\mu_{(\mathbf{x})}$, $\mu_{(\mathbf{x}_k)}$ and $\sigma_{(\mathbf{x})}$, $\sigma_{(\mathbf{x}_k)}$  are the means and variances of the original image  $\mathbf{x}$ and the restored color image $\mathbf{x}_k$, respectively, $\sigma_{\mathbf{x}\mathbf{x}_k}$ is the covariance between $\mathbf{x}$ and $\mathbf{x}_k$, $c_1=(0.01L)^2, c_2=(0.03L)^2$ with $L$ being the dynamic range of pixel values, and $d=255$.

\begin{table}[!ht]
    \centering
    \caption{ Numerical results for Example \ref{Example-4} with $l=0.001$\label{Table-8}}
    \begin{tabular}{|llllllll|}
    \hline 
        Images &~&GRK & TGRK & SRK& TSRK& SRKS &TSRKS\\ \hline
         \multirow{3}*{Brain}&PSNR & 32.6632 & 33.9135  & 33.4544 & 34.1195& 33.8799& 34.3792\\
        &SSIM & 0.9278 & 0.9342  & 0.9362&0.9410& 0.9368& 0.9379 \\
        &EN & 0.0646 & 0.0563  & 0.0590& 0.0547& 0.0563&0.0533\\ \hline
    
         \multirow{3}*{Flower}&PSNR & 27.8215 & 30.3581  & 29.5524 & 30.5970& 29.68068& 30.9137\\
        &SSIM & 0.7316 & 0.8378  & 0.8165&0.8564& 0.8171 & 0.8607 \\
        &EN & 0.0814 & 0.0608  & 0.0666& 0.0591& 0.0657&0.0570 \\ \hline
        
         \multirow{3}*{House}&PSNR & 24.1345 & 25.7946  & 25.2515 & 26.4440& 25.4077& 26.5850\\
        &SSIM & 0.5888 & 0.6662  & 0.6486&0.7029& 0.6540 & 0.7046 \\
        &EN & 0.0852 & 0.0702  & 0.0746& 0.0649& 0.0732&0.0640 \\ \hline
    \end{tabular}
\end{table}

\begin{example}\label{Example-4}{\rm Consider the grayscale images deblurring problems. Here the tested 
grayscale images of size $128\times 128$ including Brain, Flower, and House have blurred a matrix $A=A_1\otimes A_2\in {\mathbb{R}^{128^2 \times 128^2}}$ \cite{Xin-Fang,Li111}, in which ${A_1} = {({a_{ij}}^{(1)})_{1 \le i,j \le 128}}$ and ${A_2} = {({a_{ij}}^{(2)})_{1 \le i,j \le 128}}$ are the Toeplitz matrixs with their entries given by
$$
a_{ij}^{(1)}=\left\{\begin{aligned}	
&\frac{1}{\sigma\sqrt{2\pi}}\exp(-\frac{(i-j)^2}{2\sigma^2}),\ \ |i-j|\leq r\\
&0,\ \ \ \  \mbox{otherwise}
\end{aligned}
\right.\ \ \ \ \ a_{ij}^{(2)}=\left\{\begin{aligned}	
	&\frac{1}{2s-1},\ \ |i-j|\leq s\\
	&0,\ \ \ \  \mbox{otherwise}
\end{aligned}
\right.
$$
which simulates the Gaussian blur. For the purpose of comparison, we set the case Brain with $r=s=2,$ the case Flower with $r=s=4$, and the case House with $r=s=6$. We compare the six tested methods in the above examples, except for the GTRK and TRKS solvers, since their restoration quality is weak even if they consume much computing time.

After executing 60 seconds, we recorded the numerical results of all tested methods in Table \ref{Table-8}.  It follows from this table that, the PNSR and SSIM values corresponding to TGRK, TSRK, and TSRKS solvers are commonly more than GRK, SRK, and SRKS, respectively. This illustrates that the restoration quality of the proposed two-dimensional Kaczmarz methods is higher than the one-dimensional ones. Moreover, we can see that the TSRKS solver restores images with superior quality compared to the other solvers. The reason is that, by selecting the two active rows corresponding to the current largest and the second largest homogeneous residuals at each step and requiring small rows of $A$, the storage and computing time are vastly less than other solvers.
The resorted images were displayed in Figs.\ref{Fig:5}-\ref{Fig:7}, which illustrates the same conclusions.}
\end{example}

 \begin{figure}[htbp]
\begin{center}
	\begin{minipage}[c]{1.0\textwidth}
		\includegraphics[width=5in]{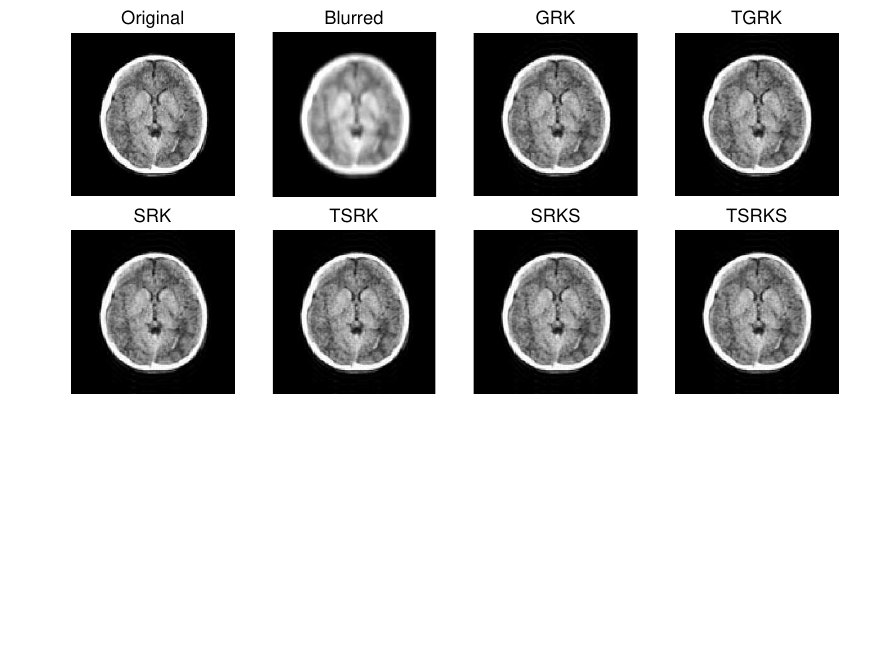}
	\end{minipage}
\end{center}
	\vspace{-10em}\caption{Brain with $r=2,s=2$ \label{Fig:5}}
\end{figure}
 \begin{figure}[htbp]
\begin{center}
	\begin{minipage}[c]{1.0\textwidth}
		\includegraphics[width=5in]{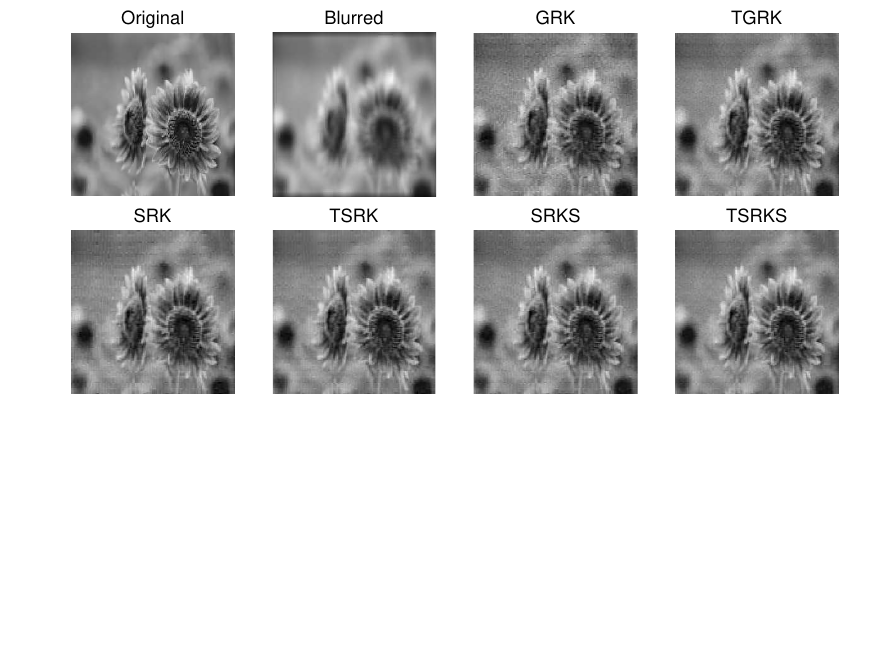}
	\end{minipage}
\end{center}
	\vspace{-10em}\caption{Flower with $r=4,s=4$ \label{Fig:6}}
\end{figure}

 \begin{figure}[htbp]
\begin{center}
	\begin{minipage}[c]{1.0\textwidth}
		\includegraphics[width=5in]{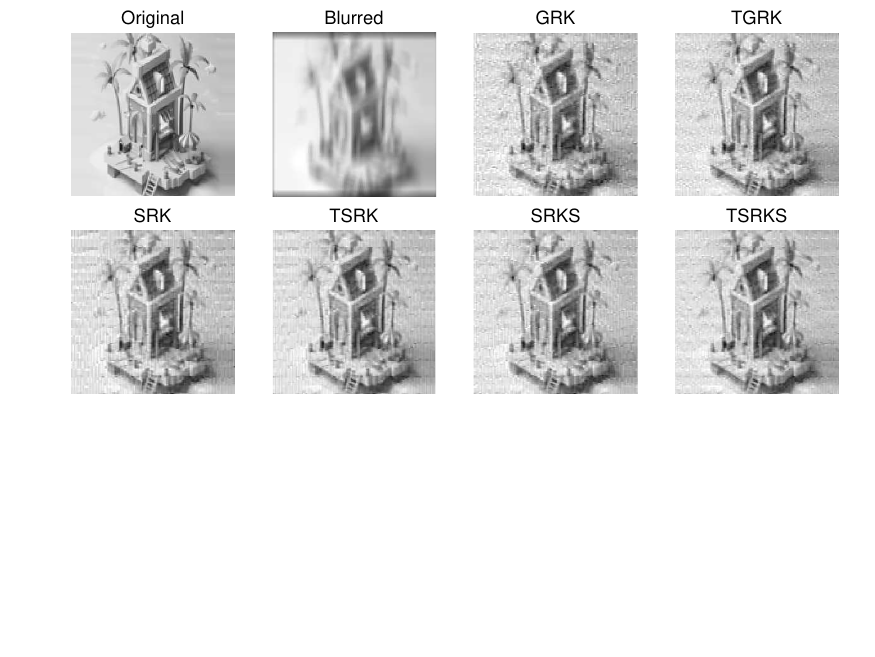}
	\end{minipage}
\end{center}
	\vspace{-10em}\caption{House with $r=6,s=6$ \label{Fig:7}}
\end{figure}

\section{Conclusions}
\label{section-5}
In this paper, we proposed the two-dimensional  Kaczmarz-type methods to solve large-scale linear systems. The main contributions are listed as follows.\\
\begin{itemize}
    \item From a pure projection viewpoint, a two-dimensional randomized Kaczmarz method is first proposed for solving large-scale linear systems, which is more straightforward and easier than the method proposed in \cite{WWT}. 
    \item By selecting two larger entries of the residual vector at each iteration, we devise a two-dimensional greedy randomized Kaczmarz method. The convergence analysis of which is also newly established.
    \item To improve the convergence rate of the mentioned-above methods, we propose a two-dimensional semi-randomized Kaczmarz method and its modified version with simple random sampling, which selects the two active rows corresponding to the current largest and the second largest homogeneous residual, without traversing all rows of the coefficient matrix in each iteration to construct the index set.
    \item  The obtained numerical results, including when the proposed solvers for some randomly generated data and image deblurring problems, illustrate the feasibility and validity of the proposed solvers compared to many state-of-the-art randomized Kaczmarz methods.
\end{itemize}


\end{document}